\RequirePackage{fix-cm}
\documentclass[smallextended]{svjour3}
\journalname{Mathematical Programming Computation}

\usepackage{graphicx}

\makeatletter\let\svjour@makecaption\@makecaption\makeatother
\usepackage{subcaption}
\makeatletter\let\@makecaption\svjour@makecaption\makeatother

\usepackage{amssymb}
\usepackage{amsmath}

\usepackage{amsthm}
\usepackage{mathtools}
\usepackage{url}
\usepackage[ruled,vlined,linesnumbered]{algorithm2e}
\usepackage{hyperref}

\newcommand{\KKT}{\mathrm{KKT}}

\DeclareMathOperator*{\diag}{diag}
\DeclareMathOperator*{\LSE}{LSE}
\DeclareMathOperator*{\R}{\mathbb{R}}

\newtheorem{prop}{Proposition}
\newtheorem{defn}{Definition}
\newtheorem{rema}{Remark}

\title{GPU-Accelerated Bregman Douglas--Rachford Splitting for Discrete Optimal Transport}
\titlerunning{GPU Bregman DR Splitting for Discrete OT}
\author{Yifan Xu \and Shiqian Ma}
\authorrunning{Y. Xu, S. Ma}
\institute{Y. Xu \at
              Department of Computational Applied Mathematics and Operations Research,
              Rice University, Houston, TX, USA \\
              \email{yx76@rice.edu}
           \and
           S. Ma (corresponding author) \at Data Science and AI Institute, and
              Department of Applied Mathematics and Statistics,
              Johns Hopkins University, Baltimore, MD, USA \\
              \email{sqma@jhu.edu}}

\date{Received: date / Accepted: date}

\begin{document}

\maketitle

\begin{abstract}
We present GPU-accelerated implementations of the Bregman Douglas--Rachford splitting method (BDRS) for discrete optimal transport problems in three input formats: an explicit cost matrix, two point clouds with a ground cost function between them, and a separable cost on a regular grid. We derive a log-domain factored form of the iteration that replaces the $m \times n$ transport plan by $O(m+n)$ numbers. The three input formats differ in how the cost matrix reaches the arithmetic. The General solver reads a stored cost matrix, the Point Cloud solver recomputes it from the coordinates, and the Grid solver evaluates the iteration and its checks by one-dimensional reductions along the grid axes. At each termination check, a c-transform gives potentials that are dual feasible in exact arithmetic, without changing the BDRS trajectory, and the solvers stop once the row residual and the duality gap with these potentials are below the tolerance.
We benchmark them against nine GPU solvers on NVIDIA A100 GPUs, recover a primal--dual pair for the original problem from every solver's output, and score all pairs by the same Karush--Kuhn--Tucker residual. We demonstrate that our implementations of BDRS provide a favorable accuracy--time trade-off against the GPU baselines in each of the three input formats.

\end{abstract}

\section{Introduction}\label{sec:intro}

In this paper, we consider GPU implementations of the Bregman Douglas--Rachford Splitting method (BDRS) \cite{ma2025bregman,scaling-convergence} for solving the discrete optimal transport (DOT) problem:

\begin{equation}
\label{eq:primal-problem}
\tag{\ensuremath{\mathcal{P}}}
\begin{aligned}
\min_{X\in\R^{m\times n}}  & \langle C, X\rangle \\
\text{s.t.}\quad & X\mathbf{1}_n = r, \\
& X^\top \mathbf{1}_m = c, \\
& X \ge 0,
\end{aligned}
\end{equation}
where $C \in \R^{m \times n}_{+}$ is the cost matrix, $r \in \Delta^{m-1}$ and $c \in \Delta^{n-1}$ are the row and column marginals, and $X$ is the transport plan. Here $\Delta^{s-1} = \{z\in\mathbb{R}^s \mid \sum_{i=1}^s z_i =1, z\geq 0\}$ and $\mathbf{1}_n$ denotes the $n$-dimensional all-one vector.

The dual problem of \eqref{eq:primal-problem} is given by:

\begin{equation}
    \label{eq:dual-problem}
    \tag{\ensuremath{\mathcal{D}}}
    \begin{aligned}
        \max_{\alpha \in \R^m, \beta \in \R^n} &\langle \alpha, r \rangle + \langle \beta, c \rangle \\
        \text{s.t.}\quad & C - \alpha\mathbf{1}_n^\top - \mathbf{1}_m\beta^\top\geq 0.
    \end{aligned}
\end{equation}
We refer to the dual variables $\alpha$ and $\beta$ as the row and column potentials, respectively.

The DOT problem~\eqref{eq:primal-problem} is a special linear programming problem. There exist various classical algorithms to solve the DOT problem, including network simplex~\cite{orlin1997network}, and specialized transportation and assignment routines such as the auction~\cite{bertsekas1988auction} and Hungarian~\cite{kuhn1955hungarian,munkres1957algorithms} algorithms. We refer to \cite{peyre2019computational} for a thorough survey of computational OT. Recently, the IIO algorithm~\cite{bargetto2026iio}, a variant of the network simplex algorithm, was specialized to image grids~\cite{bargetto2026improved} and provides exact solutions to grid DOT problems of up to $512^4$ decision variables.

The cost of the aforementioned classical algorithms often scales steeply with problem size. Therefore, there has been a rising demand for scalable approximations of the solution of DOT problems in data science. As a result, a large body of recent work solves the entropy-regularized OT problem (EOT). This includes the Sinkhorn algorithm~\cite{cuturi2013sinkhorn}
and its close variants such as \cite{altschuler2017near,schmitzer2019stabilized,alaya2019screening}. There also exist software libraries for solving EOT such as POT~\cite{flamary2021pot} and OTT-JAX~\cite{cuturi2022ott}. One issue with EOT is that it requires the regularization parameter to be very small in order to well approximate the original DOT problem, which may cause numerical instability in practice.

Over the last few years, GPU implementations of first-order methods have become very competitive solvers for large-scale linear programming and DOT. Specifically, PDLP~\cite{applegate2021pdlp} and its GPU descendants, culminating in cuPDLPx~\cite{lu2025cupdlpx}, solve general LPs using only matrix--vector products, which map naturally onto GPU hardware. The same machinery has been specialized to DOT as a structured LP: PDOT~\cite{lu2024pdot} builds a GPU-accelerated solver that adapts restarted PDHG to the DOT problem by applying its marginalization operator matrix-free, and DROT~\cite{mai2021drot} is a GPU-accelerated Douglas--Rachford splitting on the transport polytope. These solvers place no assumption on the ground cost other than its nonnegativity, at the price of forming and traversing the full $m \times n$ cost and plan, whose size becomes a key memory bottleneck as the size of the problem scales up. LAMP~\cite{burns2025lamp} instead keeps only $O(m+n)$ dual variables, through a log-averaged form of primal--dual mirror prox, and either reads the cost from an explicit matrix or, in its linear-space mode, recomputes it from point coordinates.

Existing GPU-accelerated DOT solvers can roughly be classified into three categories based on the structure of the input format of the cost, as suggested in the OTT-JAX toolbox~\cite{cuturi2022ott}:

\begin{itemize}
    \item \emph{General}: the cost is provided as an explicit matrix;
    \item \emph{Point Cloud}: point coordinates are provided with a ground cost between them;
    \item \emph{Grid}: the marginals are provided on a shared grid, with per-axis costs.
\end{itemize}

The classical and GPU first-order methods mentioned above operate on the \emph{General} format: the cost is held explicitly in memory as an $m \times n$ matrix, except for LAMP's linear-space mode, which operates on the \emph{Point Cloud} format under our classification.

Driven by specific applications, a parallel line of work assumes further structure in the problem. When the problem arrives as coordinates of points in Euclidean space together with a ground distance between them, as in 3D point-cloud and shape comparison~\cite{wu2015shapenets,xia2026halo}, color transfer~\cite{solomon2015convolutional}, and domain adaptation~\cite{courty2017optimal}, the problem falls in the \emph{Point Cloud} format. Coarse-to-fine \emph{multiscale} solvers exploit this special structure to avoid forming the full cost. Specifically, HALO~\cite{xia2026halo} builds a resolution hierarchy and solves the inner subproblems using a GPU-accelerated general LP solver, and M3S~\cite{chen2022m3s} is a multiscale entropic Sinkhorn method. In applications such as image interpolation~\cite{bonneel2011displacement} and histogram-based retrieval~\cite{rubner2000earth}, the marginals in the \emph{Grid} format are given as histograms on a shared rectangular grid and the cost is given by \emph{separable} per-axis terms. For this class of problems, flow reformulations such as~\cite{ling2007efficient} and~\cite{auricchio2018computing} open the door to the library of network flow algorithms. Building on this, HOT~\cite{zhang2025hot} rewrites squared-Euclidean grid OT as a reduced row/column-flow linear program that never forms the $m \times n$ transport plan. See Section~\ref{sec:prelim} for a formal definition of the formats.

Recently, Ma et al. proposed the BDRS algorithm \cite{ma2025bregman} for solving maximal monotone operator inclusion problems. When specialized to DOT \eqref{eq:primal-problem}, the BDRS algorithm takes the following form:
\begin{subequations}\label{eq:uv-iteration}
\begin{align}
u^k &= r / \big((X^{k-1} \odot K) v^{k-1}\big), \label{eq:uv-u}\\
v^k &= c / \big((X^{k-1} \odot K)^\top u^k\big), \label{eq:uv-v}\\
X^{k} &= \diag(u^k) (X^{k-1} \odot K) \diag(v^k), \label{eq:uv-X}
\end{align}
\end{subequations}
where $K:= e^{-C/\eta}$ is the componentwise exponential, i.e., $K_{ij}:= e^{-C_{ij}/\eta}$, $\forall i, j$, and $\eta > 0$ is a step size. Throughout this paper, for $y$ and $z$ with the same dimension, $y\odot z$ denotes their componentwise multiplication and $y/z$ denotes their componentwise division. In this paper, we consider GPU implementations of this algorithm for DOT. We demonstrate that our GPU implementations of BDRS have very promising numerical performance, although we also point out that the convergence of~\eqref{eq:uv-iteration} for DOT has not yet been fully understood. Some recent efforts aiming to understand the convergence behavior of BDRS are discussed below. Ma \cite{scaling-convergence} proved that when \eqref{eq:uv-iteration} is applied to solve the matrix scaling problem, which is equivalent to $C=0$, it is convergent under some restrictive but verifiable conditions. Ma~\cite{bdrs-can-fail} constructed a $72 \times 12$ instance of~\eqref{eq:primal-problem} and a strictly positive starting point $(X^0, v^0)$, from which, for every fixed $\eta > 0$, the iterates $X^k$ have six non-feasible sequential limits, and so fail to converge.
Hou and Toh~\cite{hou2026badmm} proved the $R$-linear convergence of an entropy Bregman ADMM with a logarithmic dual update and a multiplier step size $s \in (0,2)$. For $s = 1$, with a zero initial multiplier and matching starting points, their Bregman ADMM has the same updating scheme as~\eqref{eq:uv-iteration}. Their result assumes a feasible, strictly positive starting plan and a fixed penalty above an unknown threshold, which do not directly apply to our solvers.

{\bf Our Contributions.}
Our main contributions are summarized below.
\begin{itemize}
    \item We derive a \emph{factored representation} of the BDRS algorithm \cite{ma2025bregman}.
    We show that the logarithm of the transport plan in the BDRS iterations can be represented as
    \[
    \log X^k_{ij} = \xi^k_i + k L_{ij}  + \zeta^k_j,
    \]
    where $L_{ij} := -C_{ij} / \eta$, and the vectors $\xi^k, \zeta^k$ obey a two-term recursion. Under the factored representation, the BDRS iterations can stay \emph{matrix-free}, reducing the memory cost of the iterates from $O(mn)$ to $O(m+n)$.
    \item We provide \emph{three implementations} of BDRS to solve DOT problems, each specialized for one problem format. In particular, all three implementations adopt the factored form, but differ in how the term related to cost, $L_{ij}$, reaches the arithmetic. Our \emph{General} solver reads $L_{ij}$ from a stored $m \times n$ matrix, so the cost is the only $m \times n$ array its iterations and termination checks hold, and, for a cost supplied with $\max(C) = 1$, dense problems of up to $2^{32}$ decision variables fit in $80$~GB of VRAM (all $1368$ such instances of our benchmark at $m = n = 2^{16}$ do, at a peak of $64.8$~GiB), the largest that we are aware of among GPU-accelerated solvers of~\eqref{eq:primal-problem} for this format. Our \emph{Point Cloud} solver recomputes $L_{ij}$ from the point coordinates in registers, so the entire solve stays matrix-free and alleviates the memory constraint. Our \emph{Grid} solver exploits the separability of the cost to read $L_{ij}$ directly from one-dimensional tables, achieving the fastest runtimes among the Grid solvers of~\eqref{eq:primal-problem} we compare, with an iteration state of under $100$~MB at our largest scale.
    \item We propose the \emph{c-transform post-processing}, adopted by all three solvers. At termination checks, a c-transform of the running dual variables produces \emph{corrected dual variables}, which are feasible in exact arithmetic, at the cost of roughly one extra iteration. Then, the row residual $\|X\mathbf{1}_n - r\|_2$ of the transport plan and its duality gap with the corrected dual variables are checked against the user-provided tolerance; this test bounds the KKT residual of the triple by the tolerance plus the column discrepancy that the marginal floor introduces (Section~\ref{sec:prelim}). The corrected dual variables are not fed back into the iteration, so the BDRS trajectory is unchanged. On the instances we measured, the triple with the corrected dual variables meets the tolerance at an earlier check than the triple with the native ones, whose dual residual lags, so the solver stops after fewer iterations. Related linear-memory certificates appear in~\cite{bobrutsky2025quantization}.
    \item We assemble \emph{nine GPU-accelerated DOT solvers} from the literature, spanning all three formats, into \emph{a single benchmark}. Each solver runs its own native stopping rule, and we recover a primal--dual pair for the original problem \eqref{eq:primal-problem} from each solver's delivered state. Since the published implementations of some baselines do not provide the dual variables, we derive the required dual mappings. Every solver is thereby scored by the same KKT residual, evaluated on its primal--dual pair for~\eqref{eq:primal-problem}. The resulting accuracy--time Pareto comparison across all three input formats, on one GPU model, constitutes our main experimental contribution.
\end{itemize}

{\bf Organization.}
The remainder of the paper is organized as follows. Section~\ref{sec:prelim} defines the three problem formats, and states the KKT residual used throughout. Section~\ref{sec:algorithm} reviews the BDRS iteration for DOT and develops the c-transform post-processing used in all three BDRS solvers. Section~\ref{sec:impl} presents the factored form of the BDRS iterates and provides the implementation details of the three solvers. Section~\ref{sec:experiments} reports the cross-solver benchmark and the supporting computational studies. Section~\ref{sec:conclusion} concludes the paper and discusses some possible future directions. Appendix~\ref{app:duals} derives the primal--dual recovery for every baseline.

\section{Preliminaries}\label{sec:prelim}

\paragraph{The discrete OT problem.}
As the first step of every solve, our solvers floor the marginals internally because their log-domain updates require strict positivity. A floor $\delta_{\mathrm{m}} > 0$ is added to every entry, and then the marginal is normalized, with $\delta_{\mathrm{m}} = 10^{-12}$ by default. We write $\tilde r := (r + \delta_{\mathrm{m}}\mathbf{1}_m)/(1 + m\delta_{\mathrm{m}})$ and $\tilde c := (c + \delta_{\mathrm{m}}\mathbf{1}_n)/(1 + n\delta_{\mathrm{m}})$ for the floored marginals, and use $(r, c)$ for the instance's original marginals. The floored pair exists only for the BDRS iteration, while every termination-relevant quantity is evaluated against the original $(r, c)$. Next, we assume that $\max(C) > 0$ and normalize the cost matrix $C \gets C / \max(C)$ to ensure $\max(C) = 1$. Since the optimal solution set of \eqref{eq:primal-problem} is invariant under a positive rescaling of $C$, this establishes a common scale against which step sizes and residual magnitudes are read, without changing the primal solution $X^\star$. As a linear program with a nonempty, bounded feasible set, \eqref{eq:primal-problem} attains its optimum and strong duality holds at optimum, so its optimal value equals that of \eqref{eq:dual-problem}.

\paragraph{Problem formats.}
Section~\ref{sec:intro} classified prior solvers by the structure of the problem. We now define the three formats formally. In the definitions below, problem \emph{structure} and problem \emph{format} play distinct roles. The structure is a latent property of the problem and makes certain formats of the problem possible, while the format is that possibility materialized in memory. A solver consumes formats, not structure.

\begin{defn}[General]\label{def:general}
In the \emph{General} format, the cost $C \in \R^{m\times n}_{+}$ is a given nonnegative matrix, and the marginal distributions $r$ and $c$ are given probability distributions in their respective dimensions.
\end{defn}

\begin{defn}[Point Cloud]\label{def:pointcloud}
An instance in the \emph{Point Cloud} format is specified by points $x_1,\dots,x_m \in \R^d$ and $y_1,\dots,y_n \in \R^d$, a ground cost $\kappa\colon \R^d\times\R^d \to \R_{+}$, and marginals $r \in \Delta^{m-1}$, $c \in \Delta^{n-1}$. This induces problem~\eqref{eq:primal-problem} with the cost matrix $C_{ij} := \kappa(x_i, y_j)$, $i \in [m]$, $j \in [n]$.
\end{defn}

\begin{defn}[Grid]\label{def:grid}
  An instance of the \emph{Grid} format is specified by grid dimensions $h, w \in \mathbb{N}$, axis costs $\varphi_{\mathrm{r}}\colon \{0, 1, \dots, h-1\} \to \R_{+}$ and $\varphi_{\mathrm{c}}\colon \{0, 1, \dots, w-1\} \to \R_{+}$, and marginals $r, c \in \Delta^{hw-1}$ on the common index set $[h] \times [w]$, where $[s] :=\{1,2,\ldots,s\}$. Let $N := hw$. The provided data induce problem~\eqref{eq:primal-problem} with $m = n = N$ and the \emph{separable} cost matrix
\[
  C_{(p,q),(p',q')} := \varphi_{\mathrm{r}}(|p-p'|) + \varphi_{\mathrm{c}}(|q-q'|),
\]
indices written as row-column pairs under a fixed identification of $[h] \times [w]$ with $[N]$.
\end{defn}

These formats nest by generality: every Grid instance we consider is also a Point Cloud instance, and every Point Cloud instance is a General instance. A solver built for a broader format therefore also handles instances from a narrower one (by converting the input data into the broader format), whereas a solver that assumes more structure attains cheaper iterations and/or requires less device memory. We develop one implementation of BDRS for each format.

Table~\ref{tab:complexity} summarizes the device memory and the per-iteration FLOPs of our three solvers against the GPU baselines of each format. The Point Cloud rows are read under the assumption that each evaluation of an entry of $C$ (via computing $\kappa(x_i, y_j)$) costs $O(d)$ FLOPs, a common assumption every Point Cloud problem in this paper meets. Notably, in the General and Grid rows, our BDRS solvers (which target the original DOT problem~\eqref{eq:primal-problem}) have the same memory requirements and FLOPs per iteration as the respective entropic (Sinkhorn) solvers. This is not a coincidence, and we defer further analysis and comparison of the two algorithms to Remark~\ref{rem:compare-with-sinkhorn}.

\begin{table}[t]
\centering
\caption{Device memory requirement and per-iteration FLOPs, by format. $d$ is the dimension of the point clouds, and typically $d \ll \min(m, n)$. $s$ is a parameter dependent on data sparsity, with $s \le mn$ and typically $s \ll mn$.}
\label{tab:complexity}
\begin{tabular}{llll}
\hline
Format & Solver & Memory & FLOPs per iteration \\
\hline
General & \textbf{General (ours)} & $O(mn)$ & $O(mn)$ \\
 & DROT~\cite{mai2021drot} & $O(mn)$ & $O(mn)$ \\
 & PDOT~\cite{lu2024pdot} & $O(mn)$ & $O(mn)$ \\
 & cuPDLPx~\cite{lu2025cupdlpx} & $O(mn)$ & $O(mn)$ \\
 & Sinkhorn~\cite{cuturi2013sinkhorn,flamary2021pot} & $O(mn)$ & $O(mn)$ \\
 & LAMP~\cite{burns2025lamp} & $O(mn)$ & $O(mn)$ \\
\hline
Point Cloud & \textbf{Point Cloud (ours)} & $O((m+n)d)$ & $O(mnd)$ \\
 & HALO~\cite{xia2026halo} & $O((m+n)d + s)$ & data-dependent \\
 & M3S~\cite{chen2022m3s} & $O(m+n+s)$ & $O(s)$ \\
 & LAMP, linear space~\cite{burns2025lamp} & $O((m+n)d)$ & $O(mnd)$ \\
\hline
Grid & \textbf{Grid (ours)} & $O(N)$ & $O(N(h+w))$ \\
 & HOT~\cite{zhang2025hot} & $O(N(h+w))$ & $O(N(h+w))$ \\
 & Conv.\ Sinkhorn~\cite{solomon2015convolutional,cuturi2022ott} & $O(N)$ & $O(N(h+w))$ \\
\hline
\end{tabular}
\end{table}

\begin{rema}\label{rem:pc-meaning}
Every nonnegative $C$ of the General format is syntactically also a Point Cloud instance: with $\kappa(u,v) = (\langle u, v \rangle)_+$, $x_i = e_i$, and $y_j$ the $j$-th column of $C$, one has $C_{ij} = \kappa(x_i, y_j)$ with $d = m$; alternatively using the rows of $C$ gives $d = n$, so $d = \min(m,n)$ suffices. However, this transformation buys neither more desirable memory requirements nor a lower per-iteration complexity. If $\kappa$ is a fixed locally Lipschitz function that does not depend on $C$, the matrices $C \in \R^{m \times n}$ it represents with $(m+n)d < mn$ coordinates form a set of measure zero, so almost every $C$ requires $d \ge mn/(m+n)$, and at that dimension the format's $O((m+n)d)$ memory requirement and $O(mnd)$ per-iteration FLOPs are no better than those of the General solver. As a result, the Point Cloud format and its associated solver are designed for DOT instances whose cost has a low \emph{parametric dimension} $d$, meaning that the cost $C$ should be expressible by $O((m+n)d)$ parameters rather than $O(mn)$, with $d \ll mn/(m+n)$. In particular, $d \in \{2,3\}$ for every Point Cloud instance in this paper.
\end{rema}

\paragraph{Feasibility and Optimality.}
In the process of a solve, we track three residuals. The primal residual $r_P(X) = \|X\mathbf{1}_n - r\|_2 + \|X^\top\mathbf{1}_m - c\|_2 + \|\min(X, 0)\|_F$ measures the violation of the marginal and nonnegativity constraints. Since the BDRS iterates maintain $X \ge 0$ by construction, the third term vanishes for all BDRS iterates and our solvers never compute it. The dual residual $r_D(\alpha,\beta) = \|(\alpha \oplus \beta - C)_+\|_F$ measures the violation of dual feasibility $\alpha \oplus \beta \le C$. Here $\alpha \oplus \beta := \alpha\mathbf{1}_n^\top + \mathbf{1}_m\beta^\top \in \R^{m\times n}$. The duality gap of a primal--dual pair is $g(X,\alpha,\beta) := \langle C, X\rangle - (\langle \alpha, r\rangle + \langle \beta, c\rangle)$, which measures optimality. The inner products $\langle \alpha, r\rangle$ and $\langle \beta, c\rangle$ are summed over the entries of positive mass, i.e., with the convention $0 \cdot (\pm\infty) = 0$ here, so that a potential of $-\infty$ at a point of zero mass, which some entropic baselines return, leaves the dual objective finite. By weak duality, the dual objective of any feasible $(\alpha,\beta)$ lower-bounds the optimal cost, and by strong duality, the gap is zero if and only if the feasible primal--dual pair is also optimal. Therefore, a small gap evaluated at a feasible primal--dual pair certifies near-optimality.

\paragraph{Error metrics and stopping criteria.}
All solvers are compared on the three residuals above, each evaluated in $O(m+n)$ auxiliary memory. We aggregate them into the \emph{KKT residual} of a primal--dual triple,
\begin{equation} \label{eq:kkt}
  \KKT(X,\alpha,\beta) := \max\{ r_P(X), r_D(\alpha,\beta), |g(X,\alpha,\beta)| \}.
\end{equation}
All three of our solvers return $(X^k, \alpha', \beta')$ at a termination check when $\|X^k\mathbf{1}_n - r\|_2 < \mathrm{tol}$ and $|g(X^k, \alpha', \beta')| < \mathrm{tol}$, where $(\alpha', \beta')$ is the corrected dual of Section~\ref{sec:c-transform} and $\mathrm{tol}$ is a user-specified tolerance that defaults to $10^{-4}$. In exact arithmetic, $r_D(\alpha', \beta') = 0$ and the iteration enforces $X^{k\top}\mathbf{1}_m = \tilde c$, so this test guarantees $\KKT(X^k, \alpha', \beta') < \mathrm{tol} + \|\tilde c - c\|_2$. In floating point, both identities hold only up to rounding, and so does the bound. On every DOTmark instance in our benchmark, $\|\tilde c - c\|_2 \le 2.4 \times 10^{-9}$, so with $\mathrm{tol} = 10^{-8}$ the bound is at most $1.24 \times 10^{-8}$.

\section{BDRS with c-transform post-processing}\label{sec:algorithm}

\subsection{The Bregman Douglas--Rachford splitting iteration}\label{sec:bdrs}
Our solvers implement the BDRS method of~\cite{ma2025bregman}.
We apply it to~\eqref{eq:primal-problem} under the Boltzmann--Shannon entropy kernel $h(X) = \sum_{i,j} X_{ij}(\log X_{ij} - 1)$. The domain of $h$ enforces the constraint $X \ge 0$ throughout the iterates, and each resolvent is a Kullback--Leibler (KL) projection onto one marginal distribution, available in closed form. The algorithm BDRS for DOT is given in \eqref{eq:uv-iteration} and we refer to~\cite{ma2025bregman} for details about the algorithm.
The primal iterate $X^k \ge 0$ plays the role of the Lagrange multiplier of the dual inequality $\alpha \oplus \beta \le C$, with~\eqref{eq:uv-X} as its multiplier update, and the scalings $u^k$ and $v^k$ carry the potentials of~\eqref{eq:dual-problem} as $\alpha^k = \eta\log u^k$ and $\beta^k = \eta\log v^k$. We start from $X^0 = \tfrac{1}{mn}\mathbf{1}_m\mathbf{1}_n^\top$ and $v^{0} = \mathbf{1}_n$, running~\eqref{eq:uv-iteration} for $k = 1, 2, \dots$. The step size $\eta$ is held constant across iterations. Its deployed value is set by a fixed rule of the ground cost, given in Section~\ref{sec:experiment-setup}.

Every factor in~\eqref{eq:uv-u}--\eqref{eq:uv-X} is strictly positive, so $X^k > 0$ for all $k$ in exact arithmetic. However, the guarantee is lost in floating-point arithmetic. We implement~\eqref{eq:uv-iteration} in the log domain to address this issue, as detailed in Section~\ref{sec:impl}.

\begin{rema} \label{rem:compare-with-sinkhorn}
The iteration~\eqref{eq:uv-iteration} can be interpreted as the \emph{alternating direction exponential multiplier method} (ADEMM)~\cite{ma2025bregman}. More specifically, the~\eqref{eq:uv-u} and~\eqref{eq:uv-v} steps minimize, alternately over the two potential blocks, the Lagrangian of~\eqref{eq:dual-problem} penalized by an exponential rather than quadratic penalty function (as in the classical exponential multiplier method~\cite{bertsekas1982constrained,tseng1993exponential}), and \eqref{eq:uv-X} is the corresponding multiplier update. Moreover, we draw connections between~\eqref{eq:uv-iteration} and the Sinkhorn iterations, whose iteration on the entropically regularized problem with regularization $\varepsilon$ is
\begin{equation}\label{eq:sinkhorn-iteration}
  u^k = r /\big(K_\varepsilon v^{k-1}\big), \qquad
  v^k = c /\big(K_\varepsilon^\top u^{k}\big), \qquad
  K_\varepsilon := e^{-C/\varepsilon}.
\end{equation}
We make the observation that \eqref{eq:uv-u} and \eqref{eq:uv-v} are Sinkhorn updates with the constant kernel $K_\varepsilon$ replaced by the iteration-dependent $X^{k-1} \odot K$. However, BDRS is applied to the \emph{original} DOT~\eqref{eq:primal-problem} rather than its entropically regularized surrogate.
The solution of Sinkhorn's algorithm approaches that of~\eqref{eq:primal-problem} only as $\varepsilon \to 0$.
An iteration of BDRS has almost the same cost as a Sinkhorn iteration because~\eqref{eq:uv-iteration} pays only one additional matrix Hadamard product to obtain $X^{k-1} \odot K$.
\end{rema}

\subsection{The c-transform certificate and the termination check procedure}\label{sec:c-transform}

In our numerical experiments, we observed that the three optimality residuals in \eqref{eq:kkt} behave differently along the trajectory generated by~\eqref{eq:uv-iteration}. The primal residual $r_P^k := r_P(X^k)$ meets the tolerance $\mathrm{tol}$ in relatively few iterations. The native duality gap $g(X^k,\alpha^k,\beta^k)$ at the running iterates meets it within a comparable number of iterations. However, the dual residual $r_D^k$ does not. At the termination check where the other two first meet the tolerance, $r_D^k$ is often still orders of magnitude above them. Section~\ref{sec:exp-ablation} presents a numerical demonstration of this dual residual stalling phenomenon.

The stall matters because, with a greatly infeasible dual, a small duality gap is not a valid measure of suboptimality. The dual objective $\langle\alpha^k,r\rangle + \langle\beta^k,c\rangle$ lower-bounds the optimal value of~\eqref{eq:primal-problem} only when $(\alpha^k,\beta^k)$ is feasible. When $r_D^k$ is relatively large, the dual objective may overshoot, so the native gap $g(X^k,\alpha^k,\beta^k)$ is \emph{optimistic}.

We propose to restore dual feasibility in a \emph{one-shot} post-processing step. For a fixed row potential $\alpha$, the largest column potential that keeps dual feasibility $\alpha\oplus\beta \le C$ is the \emph{c-transform}~\cite{peyre2019computational}
\begin{equation}\label{eq:c-transform}
  \beta'_j := \min_{1\le i\le m}(C_{ij} - \alpha_i), \qquad \forall j \in [n] .
\end{equation}
By construction, we have $\alpha_i + \beta'_j \le C_{ij}$ for every $(i,j)$, so $(\alpha,\beta')$ is dual feasible up to machine precision. Since $\beta'$ is the largest feasible choice given $\alpha$, it yields the tightest gap available for that choice of $\alpha$. The symmetric transform $\alpha'_i := \min_{1\le j\le n}(C_{ij} - \beta_j)$ does the same for a given choice of $\beta$, and we apply the two sequentially: we first form $\beta'$ from $\alpha^k$ by~\eqref{eq:c-transform} and then $\alpha'_i := \min_{1\le j\le n}(C_{ij} - \beta'_j)$, so that $\beta^k$ is not used. We call $(\alpha', \beta')$ the \emph{corrected dual}.

Because $(\alpha',\beta')$ is feasible, weak duality makes $\langle\alpha',r\rangle + \langle\beta',c\rangle$ a true (though not necessarily tight) lower bound to the optimal objective value of~\eqref{eq:primal-problem}. Writing $p^\star$ for the optimal value of~\eqref{eq:primal-problem}, weak duality gives $\langle C, X^k\rangle - p^\star \le g(X^k,\alpha',\beta')$ for every $X^k$, and in addition, $0 \le \langle C, X^k\rangle - p^\star$ when $X^k$ is feasible. The duality gap of the primal iterate and the corrected dual hence bounds the optimality of a near-feasible $X^k$ only from above. Therefore, our solvers adopt the certified gap $g(X^k,\alpha',\beta')$ as opposed to the optimistic native gap $g(X^k,\alpha^k,\beta^k)$.

This gives the termination check procedure we implement in all three BDRS solvers. At a check we form $(\alpha',\beta')$ from~\eqref{eq:c-transform} and its $\alpha$-counterpart, and terminate at tolerance $\mathrm{tol}$ once $\|X^k\mathbf{1}_n - r\|_2 < \mathrm{tol}$ and $|g(X^k,\alpha',\beta')| < \mathrm{tol}$. Further, evaluating $r_P^k$ at a termination check is cheaper than its definition suggests. After the $v$-update, the column sums of $X^k$ agree with $c$ exactly, up to the flooring of the marginals, which is small relative to our tolerances:
\[
\sum_i X^k_{ij} = v^k_j\big((X^{k-1}\odot K)^\top u^k\big)_j = \tilde{c}_j \approx c_j \quad \forall j.
\]
Therefore, all three solvers' termination checks evaluate the row residual alone. The c-transform runs only at the checks where the row residual $\|X^k\mathbf{1}_n - r\|_2$ is below $\mathrm{tol}$. Termination checks are performed every $T_{\mathrm{chk}}$ iterations, with a default value dependent on problem size. Details of the $T_{\mathrm{chk}}$ scheme, as well as a numerical study on the robustness of our solvers with respect to $T_{\mathrm{chk}}$, can be found in Section~\ref{sec:ablation-robustness}.

\begin{rema}
The corrected duals $(\alpha',\beta')$ serve only to check termination. If the termination check fails and $(X^k, \alpha', \beta')$ is not accepted, then the iteration resumes from the un-transformed primal--dual iterate $(X^k, \alpha^k, \beta^k)$, not from the corrected duals.
\end{rema}

Our three solvers share the following template of abstract operations.
\begin{enumerate}
  \item \textsc{Normalize}. Rescale the cost so that $\max (C) = 1$ and floor the marginals, yielding $(C, \tilde r, \tilde c)$ as in Section~\ref{sec:prelim}. Also, initialize $(X^0, v^0) = (\tfrac{1}{mn}\mathbf{1}_m\mathbf{1}_n^\top, \mathbf{1}_n)$.
  \item \textsc{BdrsStep}. For $k = 1, 2, \dots$, advance $(X^{k-1}, v^{k-1})$ to $(X^k, v^k)$ by the BDRS iteration~\eqref{eq:uv-iteration}. Note that $u^{k-1}$ is never consumed by the $k$-th iteration, so the iteration itself carries only $(X^{k-1}, v^{k-1})$. \label{step:2}
  \item \textsc{PrimalRes}. When $k \equiv 0 \mod T_{\mathrm{chk}}$, evaluate the row residual $\|X^k\mathbf{1}_n - r\|_2$. If it is at least $\mathrm{tol}$, go back to step~\ref{step:2}.
  \item \textsc{CTransform} and \textsc{Gap}. Otherwise, form the corrected duals $(\alpha', \beta')$ by~\eqref{eq:c-transform} and its $\alpha$-counterpart, and compute the gap $g(X^k, \alpha', \beta')$. If $|g| < \mathrm{tol}$, return $(X^k, \alpha', \beta')$; otherwise, go back to step~\ref{step:2}.
\end{enumerate}
If the iteration budget $K_{\max}$ (defaults to $10^6$ in our benchmark, see Section~\ref{sec:experiment-setup}) is exhausted, the solver returns $(X^{K_{\max}}, \alpha', \beta')$ with the potentials c-transformed at the final iterate, so that the delivered solution is dual feasible whether the solver accepted the iterates or not.

\section{Implementation Details}\label{sec:impl}
In this section, we describe the implementation details of our solvers. In particular, everything described in this section stays mathematically equivalent to the BDRS with c-transform post-processing algorithm described in the previous section. We first motivate and develop the \emph{log-domain factored form} of the BDRS iteration in Subsection~\ref{sec:impl-factored}, which all three solvers carry. Then, we describe the specific details of each of our solvers, from the least structured format to the most: Subsections~\ref{sec:impl-general}, \ref{sec:impl-pointcloud}, and~\ref{sec:impl-grid} describe our implementations of BDRS for the General, Point Cloud, and Grid formats, respectively.

\subsection{The log-domain factored form}\label{sec:impl-factored}
Implementing~\eqref{eq:uv-iteration} as-is causes numerical issues. Recall that the standard IEEE floating-point representation~\cite{ieee2019fp} has a finite exponent range. The update~\eqref{eq:uv-X} is multiplicative and hence additive in the exponents: each iteration multiplies $u^k$, $X^{k-1}$, $K$ and $v^k$, so the exponent of $X^k_{ij}$ accumulates over all $k$ iterations, and the factor $K$ alone contributes $-kC_{ij}/\eta$ to it, a term unbounded below in $k$ whenever $C_{ij} > 0$. Empirically, the entries of the transport plan that decay towards zero can cross the smallest representable FP64 threshold $5 \times 10^{-324}$, in which case the entry underflows to numerical zero, and a zero entry in the transport plan stays zero under~\eqref{eq:uv-X}. Moreover, if an underflowed entry carried most of the mass of its row (resp. column), the scaling update~\eqref{eq:uv-u} (resp.~\eqref{eq:uv-v}) inflates towards infinity, and the c-transform as well as the duality gap become undefined.

To address these issues, all three solvers run in the \emph{log domain}, storing logarithms of the iterates rather than the iterates themselves. In particular, an entry that underflows FP64 at $5\times10^{-324}$ has logarithm ${\approx}-745$, a double well within the representable range, so the accumulation of~\eqref{eq:uv-X} becomes an ordinary addition of exponents.

Our next proposition shows that, in this domain, we can factor the transport plan $\log X$ into two accumulation vectors and the iteration counter.

\begin{prop}[The factored form of the log-domain BDRS iteration]\label{prop:factored}
Denote $L := \log K = - C/\eta$, and assume $r > 0$ and $c > 0$. Let $\{(u^k, v^k, X^k)\}$ be the sequence of iterates generated by~\eqref{eq:uv-iteration} with initializations $X^0 = \tfrac{1}{mn}\mathbf{1}_m\mathbf{1}_n^{\top}$, $v^0 = \mathbf{1}_n$, and define the vectors $\xi^k \in \R^m$ and $\zeta^k \in \R^n$ by
\begin{equation}\label{eq:factored-def}
\xi^k_i := -\log m + \sum_{t=1}^{k}\log u^t_i,
\qquad
\zeta^k_j := -\log n + \sum_{t=1}^{k}\log v^t_j,
\end{equation}
for all $i \in [m]$ and $j \in [n]$. Then, for every $k \ge 0$,
\begin{equation} \label{eq:factored-plan}
  \log X^k_{ij} = \xi^k_i + kL_{ij} + \zeta^k_j \qquad \forall (i, j) \in [m] \times [n],
\end{equation}
and for every $k \ge 1$, starting from $\xi_i^0 = -\log m$ and $\zeta^0_j = -\log n$, the vectors obey the recursions
\begin{equation}\label{eq:factored-iter}
\xi^k_i = \log r_i - \LSE_j\big(kL_{ij} + \tilde\beta_j\big),
\qquad
\zeta^k_j = \log c_j - \LSE_i\big(kL_{ij} + \xi^k_i\big),
\end{equation}
where $\tilde\beta_j := \zeta^{k-1}_j + \log v^{k-1}_j$ and $\LSE_i$, $\LSE_j$ denote the Log-Sum-Exp (LSE) operation over the indices $i$ and $j$ of the input matrix, respectively.
\end{prop}

\begin{proof}
  First, we prove~\eqref{eq:factored-plan} by induction on $k$. \textbf{Base case:} at $k=0$, $\xi^0_i + 0 \cdot L_{ij} + \zeta^0_j = - \log(mn) = \log X^0_{ij}$, as desired. \textbf{Inductive step:} assume that~\eqref{eq:factored-plan} holds at iteration $k-1$. Taking the logarithm of both sides of~\eqref{eq:uv-X},
  \begin{align*}
    \log X^k_{ij} &= \log u^k_i + \log X^{k-1}_{ij} + L_{ij} + \log v^k_j \\
    &= \big(\xi^{k-1}_i + \log u^k_i\big) + k L_{ij} + \big(\zeta^{k-1}_j + \log v^k_j\big) \\
    &= \xi_i^k + kL_{ij} + \zeta_j^k,
  \end{align*}
  as desired.

  Next, we prove~\eqref{eq:factored-iter} from~\eqref{eq:factored-plan}.
  Using~\eqref{eq:factored-plan} at iteration $k-1$ and taking logarithms of both sides of~\eqref{eq:uv-u},
  \begin{align*}
    \log u^k_i &= \log r_i - \log \Big( \textstyle\sum_j \exp\big( \log X^{k-1}_{ij} + L_{ij} + \log v^{k-1}_j \big) \Big) \\
    \log u^k_i &= \log r_i - \LSE_j\big( \xi^{k-1}_i + kL_{ij} + \zeta^{k-1}_j + \log v^{k-1}_j \big) \\
    \log u^k_i &= \log r_i - \xi^{k-1}_i - \LSE_j\big( kL_{ij} + \tilde\beta_j \big) \\
    \xi^{k-1}_i + \log u^k_i &= \log r_i - \LSE_j\big( kL_{ij} + \tilde\beta_j \big) \\
    \xi^k_i &= \log r_i - \LSE_j\big( kL_{ij} + \tilde\beta_j \big),
  \end{align*}
  where the second equality substitutes~\eqref{eq:factored-plan} at iteration $k-1$ into the definition of $\LSE$, and the third moves the $j$-independent term $\xi^{k-1}_i$ out of the $\LSE$. Similarly, taking logarithms of both sides of~\eqref{eq:uv-v},
  \begin{align*}
    \log v^k_j &= \log c_j - \LSE_i\big( \log X^{k-1}_{ij} + L_{ij} + \log u^k_i \big) \\
    \log v^k_j &= \log c_j - \LSE_i\big( \xi^{k-1}_i + \log u^k_i + kL_{ij} + \zeta^{k-1}_j \big) \\
    \log v^k_j &= \log c_j - \zeta^{k-1}_j - \LSE_i\big( kL_{ij} + \xi^k_i \big) \\
    \zeta^{k-1}_j + \log v^k_j &= \log c_j - \LSE_i\big( kL_{ij} + \xi^k_i \big) \\
    \zeta^k_j &= \log c_j - \LSE_i\big( kL_{ij} + \xi^k_i \big),
  \end{align*}
  as desired.
\end{proof}

\begin{rema}
Proposition~\ref{prop:factored} changes the memory required for the algorithm state. By~\eqref{eq:factored-plan}, the plan $X^k$ is determined by the two vectors $\xi^k, \zeta^k$, the counter $k$ and the fixed matrix $L$, and by~\eqref{eq:factored-iter}, the next pair $(\xi^{k+1}, \zeta^{k+1})$ is computed from the current pair and $\log v^k$, so the state a solver carries from one iteration to the next is $\xi^k$, $\zeta^k$ and $\log v^k$, for a total of $O(m+n)$ memory requirement. As a result, the iteration state is never the cause of a memory bottleneck, whatever the format.
\end{rema}

\paragraph{The log-sum-exp operation.}Both the numerical stability and the computational cost of a BDRS iteration rest on how the two $\LSE$ reductions of~\eqref{eq:factored-iter} are evaluated. We apply two standard implementation tricks together.

The first is the \emph{max shift}~\cite{blanchard2021lse}. Written directly, $\log\sum_i \exp(a_i)$ overflows once any $a_i$ exceeds $\log(1.8\times10^{308}) \approx 710$. Subtracting the maximum away from each summand,
\begin{equation}\label{eq:max-shift-LSE}
  (\LSE_i Z)_j := \max_i Z_{ij} + \log \sum_{i} \exp\big(Z_{ij} - \max_i Z_{ij}\big)
  \equiv \log \sum_{i} \exp(Z_{ij}),
\end{equation}
we make every argument of $\exp$ non-positive, so each summand lies in $(0,1]$ and the sum in $[1,m]$, inside the normal-number window with consistent relative spacing between representable numbers.

The second is the \emph{online} form~\cite{milakov2018online}. The max-shift approach~\eqref{eq:max-shift-LSE} must already know the maximum in order to aggregate the summands, so it either reads the arguments twice, once for the maximum and once for the sum, or holds all of them at once. The online form does neither. It represents any partial reduction by a \emph{state} $(\mu, s)$, with $\mu$ the largest argument the partial reduction has absorbed and $s = \sum \exp(a_i - \mu)$ the sum over those arguments, so that the value of the partial reduction is $\mu + \log s$ and a single argument $a$ is the state $(a, 1)$. Two states combine by the merge operator $\boxplus$ defined as
\begin{equation}\label{eq:lse-merge}
  (\mu_1, s_1) \boxplus (\mu_2, s_2) := \big(M, s_1 e^{\mu_1 - M} + s_2 e^{\mu_2 - M}\big), \qquad M := \max(\mu_1, \mu_2),
\end{equation}
which is associative and commutative, so a reduction may be split across the threads of a block and reassembled in any order. Absorbing one new argument $a$ into a state is the special case $(\mu, s) \boxplus (a, 1)$, in which one of the two exponentials is $e^{0} = 1$:
\begin{equation}\label{eq:online-lse}
  (\mu, s) \gets
  \begin{cases}
    \big(a, s e^{\mu - a} + 1\big), & a > \mu, \\
    \big(\mu, s + e^{a - \mu}\big), & a \le \mu.
  \end{cases}
\end{equation}
Since $\mu$ is the running maximum, every exponent in~\eqref{eq:online-lse} is non-positive and $s \in [1, t]$ after $t$ arguments, so the online form keeps the guarantee of~\eqref{eq:max-shift-LSE} while reading each argument once. We quantify the concrete effect of the difference in memory traffic in the roofline analysis below.

\paragraph{Roofline model.} On top of numerical stability, the factored form also makes the iterations computationally efficient on a GPU at scale. In the roofline model~\cite{williams2009roofline}, the binding constraint of a computational pipeline is set by its arithmetic intensity, the ratio of compute performed over the memory traffic moved. The BDRS iterations have low arithmetic intensity relative to the hardware ridge, so the model predicts that moving fewer bytes of memory translates into less runtime. Under~\eqref{eq:factored-iter}, the entire algorithm state is the vectors $\xi^k$, $\zeta^k$ and $\log v$, so no $m \times n$ array is moved between iterations. Concretely, take $m = n = N$ in FP64, so that one read or write pass over an $N \times N$ array moves $8N^2$ bytes. A direct implementation of~\eqref{eq:uv-iteration} with $X^{k-1} \odot K$ being as a single array reads it once per matrix--vector product, and once more, together with $K$, to write the update, resulting in five passes for a total of $40N^2$ bytes moved, against a multiply-add per entry in each product and three multiplies in the update, for a total of about $7N^2$ FLOPs. Under that implementation, an iteration of BDRS would yield an arithmetic intensity of about $0.2$ FLOP per byte. In the General format, by contrast, the factored form reads the cost matrix $C$ once per $\LSE$, since the online form~\eqref{eq:online-lse} touches each summand once and $kL_{ij}$ is formed in registers and immediately discarded, for a total of two passes or $16N^2$ bytes. Note that if we adopted only the max-shift form~\eqref{eq:max-shift-LSE}, each $\LSE$ would read $C$ twice for a total memory traffic of about $32N^2$ bytes instead of $16N^2$. Each entry costs one multiply to form $kL_{ij}$ from $C_{ij}$, one addition for the vector term $\tilde\beta_j$ or $\xi^k_i$, and, in the $a \le \mu$ branch of~\eqref{eq:online-lse}, a subtraction, an exponential and an addition; over the two passes that is about $10N^2$ FLOPs, for an intensity of about $0.6$ FLOP per byte. This number changes drastically for the Point Cloud and Grid formats, as the cost matrix is supplied in alternative ways. We defer the details of their roofline analysis to their respective subsections. For reference, the NVIDIA A100-SXM4-80GB used in our experiments delivers $9.7$~TFLOP/s in FP64 against $2.039$~TB/s of memory bandwidth, placing the ridge of its roofline at about $4.8$ FLOP per byte~\cite{nvidia2021a100datasheet}. Both iterations' arithmetic intensities lie below the ridge value, so the roofline model predicts that both are bandwidth bound and that their time is set by bytes moved, of which the factored form moves $2.5$ times fewer. The conclusion holds on modern data-center GPUs with full-rate FP64 arithmetic. One popular example would be the NVIDIA H100, whose ridge point is about $10.1$~\cite{choquette2023hopper}.

\subsection{The General Solver}\label{sec:impl-general}

In the General format, the cost arrives as an $m \times n$ array $C$ in device memory, and nothing is assumed about it except that $C \ge 0$ and $C$ is not the zero matrix. This forces us to pass an $m \times n$ matrix through each of our iterations, and has two direct consequences. Firstly, for certain large DOT instances, the matrix simply may not fit. Concretely, the matrix requires $8mn$ bytes in FP64, which is $34$~GB at $m = n = 2^{16}$ and $550$~GB at $2^{18}$, the latter well beyond the typical $80$~GB VRAM of a server GPU. Secondly, anytime the iterations need an entry from the cost matrix, it will always cost a read. The two design questions for this format are therefore how to keep $C$ the \emph{only} $m \times n$ array the iterations hold, and how to read it as few times per iteration as possible. The factored form answers both questions.

\paragraph{Representation.} The General solver carries the factored state of Proposition~\ref{prop:factored} in the vectors $\xi^k \in \R^m$, $\zeta^k \in \R^n$ and $\log v^k \in \R^n$, and supplies the cost term of the recursions~\eqref{eq:factored-iter} by reading $C_{ij}$ from the stored array and forming $kL_{ij} = -kC_{ij}/\eta$ in registers, one multiply per entry, at the moment the $\LSE$ consumes it. Neither $K$, nor $L$, nor any iterate is stored as an $m \times n$ array, so $C$ is the only $m \times n$ array the iterations hold. Each iteration streams $C$ exactly twice, once per $\LSE$ of~\eqref{eq:factored-iter}, which results in the $16N^2$ bytes of memory traffic and $0.6$ FLOP per byte arithmetic intensity of Section~\ref{sec:impl-factored}, so that the roofline model predicts this solver to be bandwidth bound. The plan is delivered in the factored form $(\xi, \zeta, k)$ and recovered through~\eqref{eq:factored-plan} only when a dense $X$ is requested. In particular, our entire termination check procedure does not require the transport plan to be materialized.

\begin{algorithm}[t]
\caption{BDRS with c-transform for DOT in the General format}\label{alg:general}
\KwIn{cost $C\in\R^{m\times n}_{\ge 0}$; marginals $r\in\Delta^{m-1}$, $c\in\Delta^{n-1}$; step size $\eta>0$; marginal floor $\delta_{\mathrm{m}}$; tolerance $\mathrm{tol}$; check period $T_{\mathrm{chk}}$; iteration budget $K_{\max}$}
\KwOut{factored form of the plan $(\xi,\zeta,k)$ and dual-feasible $(\alpha',\beta')$}
$C \gets C / \max_{ij} C_{ij}$, $\tilde r \gets (r + \delta_{\mathrm{m}}\mathbf{1}_m)/(1 + m\delta_{\mathrm{m}})$;\enspace $\tilde c \gets (c + \delta_{\mathrm{m}}\mathbf{1}_n)/(1 + n\delta_{\mathrm{m}})$\nllabel{ln:gen-normalize}\tcp*[r]{\textsc{Normalize}}
$\xi\gets-\log m\,\mathbf{1}_m$;\enspace $\zeta\gets-\log n\,\mathbf{1}_n$;\enspace $\log v\gets\mathbf{0}_n$\nllabel{ln:gen-init}\tcp*[r]{$(X^0,v^0)$ in the factored form~\eqref{eq:factored-def}}
\For{$k = 1, 2, \dots, K_{\max}$}{
  $\tilde\beta\gets \zeta+\log v$;\quad $\xi^{\mathrm{old}}\gets \xi$\;
  $\xi_i \gets \log\tilde r_i - \LSE_j\big({-}kC_{ij}/\eta + \tilde\beta_j\big)$\nllabel{ln:gen-ulse}\tcp*[r]{$\xi$-update of~\eqref{eq:factored-iter}}
  $\zeta^{\mathrm{old}}\gets \zeta$;\quad $\zeta_j \gets \log\tilde c_j - \LSE_i\big({-}kC_{ij}/\eta + \xi_i\big)$\nllabel{ln:gen-vlse}\tcp*[r]{$\zeta$-update of~\eqref{eq:factored-iter}}
  $\log v\gets \zeta-\zeta^{\mathrm{old}}$\nllabel{ln:gen-advance}\;
  \If{$k \bmod T_{\mathrm{chk}} = 0$}{
    $\rho_i \gets \exp\big(\xi_i + \LSE_j({-}kC_{ij}/\eta + \zeta_j)\big)$\nllabel{ln:gen-proxy}\tcp*[r]{\textsc{PrimalRes}}
    \If{$\|\rho - r\|_2 < \mathrm{tol}$}{
      $\alpha \gets \eta(\xi - \xi^{\mathrm{old}})$;\enspace $\beta'_j \gets \min_i(C_{ij} - \alpha_i)$;\enspace $\alpha'_i \gets \min_j(C_{ij} - \beta'_j)$\nllabel{ln:gen-ct}\tcp*[r]{\textsc{CTransform}}
      $\texttt{obj} \gets \textstyle\sum_{ij} C_{ij}\exp\big(\xi_i - kC_{ij}/\eta + \zeta_j\big)$\nllabel{ln:gen-obj}\tcp*[r]{primal objective $\langle C, X^k\rangle$}
      \If{$|\texttt{obj} - \langle\alpha',r\rangle - \langle\beta',c\rangle| < \mathrm{tol}$}{
        \Return $(\xi, \zeta, k, \alpha', \beta')$\tcp*[r]{factored representation of $X$, via~\eqref{eq:factored-plan}}
      }
    }
  }
}
$\alpha \gets \eta(\xi - \xi^{\mathrm{old}})$;\enspace $\beta'_j \gets \min_i(C_{ij} - \alpha_i)$;\enspace $\alpha'_i \gets \min_j(C_{ij} - \beta'_j)$\tcp*[r]{feasible dual for the returned iterate}
\Return $(\xi, \zeta, k, \alpha',\beta')$\tcp*[r]{budget exhausted; plan recovered via~\eqref{eq:factored-plan}}
\end{algorithm}

\paragraph{Initialization and termination.} The remaining steps of the template of Section~\ref{sec:c-transform} run on the same factored form. \textsc{Normalize} takes the maximum of $C$ in one pass and rescales $C$. The marginals are floored by adding $\delta_{\mathrm{m}}\mathbf{1}$ and renormalizing, giving the $(\tilde r, \tilde c)$ the iteration consumes; and the initial iterates $(X^0, v^0) = (\tfrac{1}{mn}\mathbf{1}_m\mathbf{1}_n^\top, \mathbf{1}_n)$ map to $\xi = -(\log m)\mathbf{1}_m$, $\zeta = -(\log n)\mathbf{1}_n$ and $\log v = \mathbf{0}_n$ by~\eqref{eq:factored-def}. \textsc{PrimalRes} uses the row marginal error as its proxy and requires one more row $\LSE$ over $C$ (line~\ref{ln:gen-proxy} of Algorithm~\ref{alg:general}), and the full termination check proceeds only if $\|\rho - r\|_2 < \mathrm{tol}$. \textsc{CTransform} is two reads over $C$ that maintain two running minima arrays in $O(m+n)$ workspace, a column-wise one for $\beta'_j = \min_i(C_{ij} - \alpha_i)$ and a row-wise one for $\alpha'_i = \min_j(C_{ij} - \beta'_j)$, with $\alpha = \eta(\xi - \xi^{\mathrm{old}})$. Lastly, the computationally costly term of \textsc{Gap} is computing the primal objective $\langle C, X^k\rangle = \sum_{ij} C_{ij}\exp(\xi_i - kC_{ij}/\eta + \zeta_j)$, which corresponds to one read over $C$ in chunks. A check that passes the proxy therefore reads $C$ four times, the traffic of two iterations, and a check that fails it reads $C$ once. After the stop, still inside the timed solve, the solver evaluates the full primal residual $\|X^k\mathbf{1}_n - r\|_2 + \|(X^k)^\top\mathbf{1}_m - c\|_2$ of the returned iterate for the record, one row and one column $\LSE$ and thus two more reads of $C$.

\paragraph{GPU implementation.} The main loop of Algorithm~\ref{alg:general} is implemented as three CUDA kernels and a counter increment, for a total of four graph nodes, captured once per solve into a CUDA graph and replayed once per iteration. The $\xi$-update of line~\ref{ln:gen-ulse} assigns one thread block of $256$ threads to each row $i$. We store $C$ in the row-major layout where the row is contiguous, so the threads of a warp read adjacent addresses. Each thread strides along the row, forms $-kC_{ij}/\eta + \tilde\beta_j$ in registers from the entry it has just read, with $k$ taken from a device-resident counter, and aggregates it into its own online-$\LSE$ state~\eqref{eq:online-lse}. The $256$ states are aggregated in shared memory by~\eqref{eq:lse-merge}. One thread then computes $\xi_i^{\mathrm{new}} = \log \tilde r_i - (\mu + \log s)$, stores the difference $\log u_i = \xi_i^{\mathrm{new}} - \xi_i$, which the c-transform will need later, and then sets $\xi_i \gets \xi_i^{\mathrm{new}}$. The $\zeta$-update of line~\ref{ln:gen-vlse} runs against the row-major layout, and requires a more careful treatment. Our implementation runs in two stages. In the first stage, each thread owns one column over a chunk of $1024$ rows ($128$ when $m \le 4096$) and keeps one online state; a block's threads take adjacent columns and move down the rows in lockstep, so every read is again coalesced, and the partial states of all chunks are written to two scratch arrays with one entry per chunk and column ($64$~MiB in total at $m = n = 2^{16}$). In the second stage, one thread per column merges its chunk states by~\eqref{eq:lse-merge} into $(\mu, s)$, computes $\zeta_j^{\mathrm{new}} = \log \tilde c_j - (\mu + \log s)$, stores $\log v_j = \zeta_j^{\mathrm{new}} - \zeta_j$, and sets $\zeta_j \gets \zeta_j^{\mathrm{new}}$. The termination checks run as chunked passes over $C$ with PyTorch's native kernels~\cite{paszke2019pytorch}. The problem data $C$, $r$, and $c$ are loaded into device memory once, while the iteration, checks, and delivery execute on the device. At a check, only scalar verdicts return to the host.

\subsection{The Point Cloud Solver}\label{sec:impl-pointcloud}

The Point Cloud solver consumes a ground cost function $\kappa$ on two point clouds $P=(p_1, \dots, p_m)$ and $Q=(q_1, \dots, q_n)$ in $\R^d$, with the (normalized) cost matrix defined as
\[
C_{ij} := \kappa(p_i, q_j) / C_{\max},
\]
where $C_{\max} := \max_{p \in P, q \in Q} \kappa(p, q)$ is the maximum cost between the two clouds. In particular, the cost matrix is given implicitly and compactly, and we take extra care to not materialize the full $m \times n$ cost matrix at any point in the Point Cloud solve.

\paragraph{Representation.} The Point Cloud solver carries the same factored state of Proposition~\ref{prop:factored} as the other two solvers, in the vectors $\xi^k \in \R^m$, $\zeta^k \in \R^n$ and $\log v \in \R^n$. The transport plan is recoverable via~\eqref{eq:factored-plan} upon final delivery, while the term related to the cost matrix $L_{ij} = -\kappa(p_i, q_j)/(C_{\max} \eta)$ is recomputed from the coordinates in registers each time it is needed. This implies that the algorithm state occupies $O(m+n)$ memory, while the overall device memory requirement is $O((m + n)d)$ due to the storage required to store the coordinates of the point clouds.

\begin{algorithm}[t]
  \caption{BDRS with c-transform for DOT in the Point Cloud format}\label{alg:pointcloud}
  \KwIn{Point clouds $P\in\R^{m\times d}$, $Q\in\R^{n\times d}$ with ground cost $\kappa$; marginals $r\in\Delta^{m-1}$, $c\in\Delta^{n-1}$; step size $\eta>0$; marginal floor $\delta_{\mathrm{m}}$; tolerance $\mathrm{tol}$; check period $T_{\mathrm{chk}}$; budget $K_{\max}$}
\KwOut{Factored form of plan $(\xi,\zeta,k)$ and dual-feasible $(\alpha',\beta')$}
$C_{\max}\gets\max_{i,j}\kappa(p_i,q_j)$;\enspace $\tilde r \gets (r + \delta_{\mathrm{m}}\mathbf{1}_m)/(1 + m\delta_{\mathrm{m}})$;\enspace $\tilde c \gets (c + \delta_{\mathrm{m}}\mathbf{1}_n)/(1 + n\delta_{\mathrm{m}})$\tcp*[r]{\textsc{Normalize}; costs enter as $\kappa(p_i,q_j)/C_{\max} \le 1$}
$\xi\gets-(\log m)\mathbf 1_m$;\enspace $\zeta\gets-(\log n)\mathbf 1_n$;\enspace $\log v\gets\mathbf 0_n$\;
\For{$k=1,2,\dots,K_{\max}$}{
  $\tilde\beta\gets \zeta+\log v$;\quad $\xi^{\mathrm{old}}\gets \xi$\;
  $\xi_i\gets\log \tilde r_i-\LSE_j\big({-}k\kappa(p_i,q_j)/(C_{\max}\eta)+\tilde\beta_j\big)$\tcp*[r]{$\xi$-update of~\eqref{eq:factored-iter}}
  $\zeta^{\mathrm{old}}\gets \zeta$;\quad $\zeta_j\gets\log \tilde c_j-\LSE_i\big({-}k\kappa(p_i,q_j)/(C_{\max}\eta)+\xi_i\big)$\tcp*[r]{$\zeta$-update of~\eqref{eq:factored-iter}}
  $\log v\gets \zeta-\zeta^{\mathrm{old}}$\;
  \If{$k\bmod T_{\mathrm{chk}}=0$}{
    $\rho_i\gets\exp\big(\xi_i+\LSE_j({-}k\kappa(p_i,q_j)/(C_{\max}\eta)+\zeta_j)\big)$\tcp*[r]{\textsc{PrimalRes}}
    \If{$\|\rho-r\|_2<\mathrm{tol}$}{
      $\alpha\gets\eta(\xi-\xi^{\mathrm{old}})$;\enspace $\beta'_j\gets\min_i\big(\kappa(p_i,q_j)/C_{\max}-\alpha_i\big)$;\enspace $\alpha'_i\gets\min_j\big(\kappa(p_i,q_j)/C_{\max}-\beta'_j\big)$\tcp*[r]{\textsc{CTransform}}
      $\texttt{obj}\gets\textstyle\sum_{ij}\frac{\kappa(p_i,q_j)}{C_{\max}}\exp\big(\xi_i-k\kappa(p_i,q_j)/(C_{\max}\eta)+\zeta_j\big)$\tcp*[r]{primal objective $\langle C,X^k\rangle$}
      \If{$|\texttt{obj}-\langle\alpha',r\rangle-\langle\beta',c\rangle|<\mathrm{tol}$}{
        \Return $(\xi, \zeta, k, \alpha', \beta')$\tcp*[r]{factored representation of $X$, recoverable via~\eqref{eq:factored-plan}}
      }
    }
  }
}
$\alpha\gets\eta(\xi-\xi^{\mathrm{old}})$; $\beta'_j\gets\min_i\big(\kappa(p_i,q_j)/C_{\max}-\alpha_i\big)$; $\alpha'_i\gets\min_j\big(\kappa(p_i,q_j)/C_{\max}-\beta'_j\big)$\tcp*[r]{feasible duals for the returned iterate}
\Return $(\xi, \zeta, k, \alpha', \beta')$\tcp*[r]{budget exhausted, return final iterate}
\end{algorithm}

\paragraph{Initialization and termination.} In the Point Cloud format, \textsc{Normalize} is a one-time pairwise maximum $C_{\max}$ defined above, computable in $O(mnd)$ FLOPs. The marginal distribution flooring procedure is identical to the one in the General solver. The initial iterates $(X^0, v^0)$ map to the factored state as in the General solver, in $O(m+n)$. \textsc{PrimalRes}, for which we use the row marginal error as a proxy, requires a row-$\LSE$ in $O(mnd)$ FLOPs. Next, with $\alpha^k = \eta \log u^k = \eta(\xi^k - \xi^{k-1})$, the \textsc{CTransform} simply sweeps $\beta'_j=\min_i(C_{ij}-\alpha^k_i)$ and $\alpha'_i=\min_j(C_{ij}-\beta'_j)$ while recomputing $C_{ij}$ from coordinates, in $O(mnd)$ FLOPs using $O(m+n)$ workspace. Lastly, the primal objective, which is the compute-intensive part of \textsc{Gap}, is computed via exponentiating and aggregating $\xi_i - k\kappa(p_i, q_j)/(C_{\max} \eta) + \zeta_j$ across all pairs of points while weighting them with the cost $\kappa(p_i, q_j) /C_{\max}$, again recomputed from the coordinates in $O(mnd)$ FLOPs. A check that passes the proxy therefore costs four $O(mnd)$ passes, the arithmetic of two iterations, while a check that fails the proxy costs one.

\paragraph{GPU implementation.} Algorithm~\ref{alg:pointcloud} is a description of our Point Cloud solver, which implements the template of Section~\ref{sec:c-transform} for the Point Cloud format. Both row and column $\LSE$s in~\eqref{eq:factored-iter} are aggregations of $mn$ terms, and their respective kernels assign one thread block per output entry. For the $\xi$-update, the block of row $i$ loads the point $p_i$ once, and its $256$ threads iterate through the $n$ opposing points. Each thread recomputes $\kappa(p_i, q_j)$ from the coordinates in registers, forms the term $-k\kappa(p_i, q_j)/(C_{\max}\eta) + \tilde\beta_j$, and updates its private online-$\LSE$ state~\eqref{eq:online-lse}. The $256$ states in the thread block are then merged in shared memory by~\eqref{eq:lse-merge}, and the $\zeta$-update kernel is constructed similarly, over columns. Since both updates stream the opposing cloud as one contiguous array of points, there is no layout asymmetry and hence no special treatment is needed for the column $\LSE$, unlike its General counterpart. The point cloud coordinates are staged from the host onto the device once, and at termination checks only scalar verdicts return to the host. As a result, this traffic does not limit the wall-clock time. An iteration launches three kernels, the $\xi$-update, the $\zeta$-update, and a counter advance that increments the device-side iteration count $k$; the three are captured once per solve as a CUDA graph and replayed each iteration.

\begin{rema}\label{rem:pc-roofline}
In the Point Cloud solver, the arithmetic is $O(mnd)$ FLOPs per update against $O((m+n)d)$ bytes moved, making the arithmetic intensity of Point Cloud kernels grow like $O(mn/(m+n))$ with problem size and the kernels compute bound in the roofline model even on small DOT instances. In the accounting of Section~\ref{sec:impl-factored}, the cost entries are recomputed rather than read. Evaluating $\kappa(p_i, q_j)$ from the coordinates takes about $3d$ operations and replaces the read. The kernel then divides it by $C_{\max}$, as a multiply by the precomputed reciprocal $1/C_{\max}$ followed by two fused multiply-adds that restore the correctly rounded quotient, and multiplies the result by $-k/\eta$, while the addition of the vector term and the three operations of the $\LSE$ update are as before, for a total of about $3d+8$ per entry per pass, or $(6d+16)N^2$ FLOPs per iteration with $m = n = N$. Note that the opposing point cloud occupies $8Nd$ bytes, which is at most a few megabytes for the largest problems in our benchmark ($1.6$~MB at $N = 2^{16}$ and $d = 3$), so it fits in the $40$~MB L2 cache~\cite{choquette2021a100}\footnote{The size of the L2 cache may vary between devices, but the point cloud coordinates should generally fit in the L2 cache of common server GPUs.} and we assume that the repeated reads of it and of the vector term ($8N$ bytes) by the thread blocks are served from L2 rather than HBM. The $\xi$-update reads the two cloud coordinates, $\tilde\beta$, $\log\tilde r$ and the old $\xi$ and writes $\xi$ and $\log u$, while the $\zeta$-update reads the two cloud coordinates, $\xi$, $\log\tilde c$ and the old $\zeta$ and writes $\zeta$, $\log v$ and the next $\tilde\beta$, for a total of $8N(4d+11)$ bytes per iteration. The estimated arithmetic intensity is therefore about $0.18N$ FLOP per byte at $d \in \{2, 3\}$, crossing the A100-SXM4-80GB ridge at $N \approx 26$ and growing linearly in $N$. Note that the count treats the exponential as one operation, while on the device it typically consists of a sequence of basic arithmetic operations, so the true arithmetic intensity is higher still. In either case, the roofline model predicts that our Point Cloud kernels are compute bound.
\end{rema}

\subsection{The Grid Solver}\label{sec:impl-grid}

\begin{algorithm}[t]
\caption{BDRS with c-transform for DOT in the Grid format}\label{alg:grid}
\KwIn{axis costs $\varphi_{\mathrm{r}}, \varphi_{\mathrm{c}}$; marginals $r,c\in\Delta^{N-1}$ on an $h \times w$ grid; step size $\eta>0$; marginal floor $\delta_{\mathrm{m}}$; tolerance $\mathrm{tol}$; check period $T_{\mathrm{chk}}$; budget $K_{\max}$}
\KwOut{Factored form of plan $(\xi,\zeta,k)$ and dual-feasible $(\alpha',\beta')$}
$C_{\max} \gets \max_{\delta_r} \varphi_{\mathrm{r}}(\delta_r) + \max_{\delta_c} \varphi_{\mathrm{c}}(\delta_c)$\;
$\ell_{\mathrm{r}}(\delta)\gets-\varphi_{\mathrm{r}}(\delta)/(C_{\max}\eta)$;\quad $\ell_{\mathrm{c}}(\delta)\gets-\varphi_{\mathrm{c}}(\delta)/(C_{\max}\eta)$\tcp*[r]{axis tables}
$\tilde r \gets (r + \delta_{\mathrm{m}}\mathbf{1}_N)/(1 + N\delta_{\mathrm{m}})$;\enspace $\tilde c \gets (c + \delta_{\mathrm{m}}\mathbf{1}_N)/(1 + N\delta_{\mathrm{m}})$\tcp*[r]{\textsc{Normalize}: marginal floor}
$\xi\gets-(\log N)\mathbf 1_N$;\enspace $\zeta\gets-(\log N)\mathbf 1_N$;\enspace $\log v\gets\mathbf 0_N$\;
\For{$k=1,2,\dots,K_{\max}$}{
  $\tilde\beta\gets \zeta+\log v$;\quad $\xi^{\mathrm{old}}\gets \xi$\;
  $\log S_{i_r,j_c}\gets\LSE_{j_r}\big(k\ell_{\mathrm{r}}(|i_r-j_r|)+\tilde\beta_{j_r,j_c}\big)$\tcp*[r]{$\xi$-update inner pass}
  $\xi_{i_r,i_c}\gets\log \tilde r_{i_r,i_c}-\LSE_{j_c}\big(k\ell_{\mathrm{c}}(|i_c-j_c|)+\log S_{i_r,j_c}\big)$\tcp*[r]{outer pass}
  $\log T_{j_r,i_c}\gets\LSE_{i_r}\big(k\ell_{\mathrm{r}}(|i_r-j_r|)+\xi_{i_r,i_c}\big)$\tcp*[r]{$\zeta$-update inner pass}
  $\zeta^{\mathrm{old}}\gets \zeta$;\quad $\zeta_{j_r,j_c}\gets\log \tilde c_{j_r,j_c}-\LSE_{i_c}\big(k\ell_{\mathrm{c}}(|i_c-j_c|)+\log T_{j_r,i_c}\big)$\tcp*[r]{outer pass}
  $\log v\gets \zeta-\zeta^{\mathrm{old}}$\;
  \If{$k\bmod T_{\mathrm{chk}}=0$}{
    $\log S_{i_r,j_c}\gets\LSE_{j_r}\big(k\ell_{\mathrm{r}}(|i_r-j_r|)+\zeta_{j_r,j_c}\big)$\tcp*[r]{\textsc{PrimalRes} inner pass}
    $\rho_{i_r,i_c}\gets\exp\big(\xi_{i_r,i_c}+\LSE_{j_c}(k\ell_{\mathrm{c}}(|i_c-j_c|)+\log S_{i_r,j_c})\big)$\tcp*[r]{outer pass}
    \If{$\|\rho-r\|_2<\mathrm{tol}$}{
      $\alpha\gets\eta(\xi-\xi^{\mathrm{old}})$\tcp*[r]{native dual $\alpha^k=\eta\log u^k$}
      $\tau_{j_r,i_c}\gets\min_{i_r}\big(\varphi_{\mathrm{r}}(|i_r-j_r|)/C_{\max}-\alpha_{i_r,i_c}\big)$\tcp*[r]{\textsc{CTransform} $\beta'$ inner pass} \label{ln:sep-ct-inner}
      $\beta'_{j_r,j_c}\gets\min_{i_c}\big(\varphi_{\mathrm{c}}(|i_c-j_c|)/C_{\max}+\tau_{j_r,i_c}\big)$\tcp*[r]{$\beta'$ outer pass} \label{ln:sep-ct-outer}
      $\tau'_{i_r,j_c}\gets\min_{j_r}\big(\varphi_{\mathrm{r}}(|i_r-j_r|)/C_{\max}-\beta'_{j_r,j_c}\big)$\tcp*[r]{$\alpha'$ inner pass}
      $\alpha'_{i_r,i_c}\gets\min_{j_c}\big(\varphi_{\mathrm{c}}(|i_c-j_c|)/C_{\max}+\tau'_{i_r,j_c}\big)$\tcp*[r]{$\alpha'$ outer pass}
      $\log S'_{i_c,j_r}\gets\LSE_{j_c}\big(k\ell_{\mathrm{c}}(|i_c-j_c|)+\zeta_{j_r,j_c}\big)$\tcp*[r]{\texttt{obj} inner pass}
      $R^{\mathrm{r}}_{i_r,j_r}\gets\exp\big(k\ell_{\mathrm{r}}(|i_r-j_r|)+\LSE_{i_c}(\xi_{i_r,i_c}+\log S'_{i_c,j_r})\big)$\tcp*[r]{row-axis marginal of $X^k$}
      $R^{\mathrm{c}}_{i_c,j_c}\gets\exp\big(k\ell_{\mathrm{c}}(|i_c-j_c|)+\LSE_{i_r}(\xi_{i_r,i_c}+\log S_{i_r,j_c})\big)$\tcp*[r]{column-axis marginal; reuses $\log S$}
      $\texttt{obj}\gets\sum_{i_r,j_r}\varphi_{\mathrm{r}}(|i_r-j_r|)R^{\mathrm{r}}_{i_r,j_r}/C_{\max}+\sum_{i_c,j_c}\varphi_{\mathrm{c}}(|i_c-j_c|)R^{\mathrm{c}}_{i_c,j_c}/C_{\max}$\;
      \lIf{$|\texttt{obj}-\langle\alpha',r\rangle-\langle\beta',c\rangle|<\mathrm{tol}$}{\Return $(\xi, \zeta, k, \alpha', \beta')$}
    }
  }
}
$\alpha\gets\eta(\xi-\xi^{\mathrm{old}})$;\enspace $\beta'\gets$ separable c-transform of $\alpha$, see lines~\ref{ln:sep-ct-inner} and~\ref{ln:sep-ct-outer};\enspace $\alpha'\gets$ likewise of $\beta'$\;
\Return $(\xi, \zeta, k, \alpha', \beta')$\tcp*[r]{budget exhausted, return final iterate}
\end{algorithm}

The Grid solver consumes the problem~\eqref{eq:primal-problem} in the Grid format of Definition~\ref{def:grid}, and is best understood as the Point Cloud solver under two further assumptions.
\begin{enumerate}
  \item The two point clouds coincide with the same $h \times w$ grid, so that $m = n = hw =: N$, and a point $p_i$ (resp.\ $q_j$) is identified by its row and column indices $(i_r, i_c)$ (resp.\ $(j_r, j_c)$).
  \item The ground cost is \emph{separable} with translation-invariant axis costs, as in Definition~\ref{def:grid}. In other words, $\kappa(p_i, q_j)$ can be written as $\varphi_{\mathrm{r}}(|i_r-j_r|)+\varphi_{\mathrm{c}}(|i_c-j_c|)$. For example, both $\kappa_1(p_i, q_j) := \ell_1(p_i-q_j)$ and $\kappa_2(p_i, q_j) := \ell_2^2(p_i - q_j)$ are separable, with $\varphi_{\mathrm{r}, 1}(\delta) = \varphi_{\mathrm{c}, 1}(\delta) = \delta$ and $\varphi_{\mathrm{r}, 2}(\delta) = \varphi_{\mathrm{c}, 2}(\delta) = \delta^2$, respectively.
\end{enumerate}

\paragraph{Representation.} The Grid solver adopts the same factored form of Proposition~\ref{prop:factored} and stores the algorithm state in the vectors $\xi^k,\zeta^k, \log v \in \R^N$. The main difference in representation lies in how we obtain the entries of $L$, the term related to the cost matrix $C$. In particular, from the definition of $L$ (in Proposition~\ref{prop:factored}),
\begin{align*}
  L_{ij} &= - C_{ij} / \eta \\
  &= -\kappa(p_i, q_j) / (\eta C_{\max}) \\
  &= -(\varphi_{\mathrm{r}}(|i_r-j_r|)+\varphi_{\mathrm{c}}(|i_c-j_c|)) / (\eta C_{\max}) \\
  &= \ell_{\mathrm{r}}(|i_r - j_r|) + \ell_{\mathrm{c}}(|i_c - j_c|),
\end{align*}
where $\ell_{\mathrm{r}}: \{0, 1, \dots, h-1\} \to \mathbb{R}_{\leq 0}$ is defined as
\[
\ell_{\mathrm{r}} (\delta) := -\frac{1}{\eta C_{\max}}\varphi_{\mathrm{r}}(\delta),
\]
and $\ell_{\mathrm{c}}$ is defined similarly on $\{0, 1, \dots, w-1\}$. Therefore, instead of recomputing $L_{ij} = -\kappa(p_i, q_j)/(\eta C_{\max})$ from the coordinates as seen in the Point Cloud case, we can alternatively pre-compute and store $\ell_{\mathrm{r}}$ and $\ell_{\mathrm{c}}$ as two one-dimensional tables in device memory.
Moreover, the normalizer is
\begin{equation} \label{eq:grid-cmax}
C_{\max} = \max_{\delta_r} \varphi_{\mathrm{r}}(\delta_r) + \max_{\delta_c} \varphi_{\mathrm{c}}(\delta_c),
\end{equation}
which, for nondecreasing axis costs (e.g., $\ell_1$ and $\ell_2^2$ costs), reduces to the corner-to-corner displacement $\varphi_{\mathrm{r}}(h-1) + \varphi_{\mathrm{c}}(w-1)$. As a result, the device memory requirement is $O(N)$. The termination checks evaluate their axis-reduction intermediate arrays in blocks of rows, so they also stay within $O(N)$ memory. As in the Point Cloud solver, the plan is delivered in the factored form $(\xi, \zeta, k)$ at termination and recovered via~\eqref{eq:factored-plan} if needed.

The iteration evaluates the recursions~\eqref{eq:factored-iter} on this representation through a separable factorization of each $\LSE$ along the two axes of the grid. In particular, for the $\xi$-update, we aggregate across the target points $j = (j_r, j_c) \in [h] \times [w]$:
\begin{equation}\label{eq:sep-lse}
  \LSE_j\big(kL_{ij}+\tilde\beta_j\big)=\LSE_{j_c}\Big(k\ell_{\mathrm{c}}(|i_c-j_c|)+\underbrace{\LSE_{j_r}\big(k\ell_{\mathrm{r}}(|i_r-j_r|)+\tilde\beta_{j_r,j_c}\big)}_{\log S_{i_r,j_c}}\Big),
\end{equation}
which is an inner pass over rows, producing the intermediate table $\log S$, and an outer pass over columns, and likewise for the $\zeta$-update. The factorization of the $\LSE$ operation comes from the fact that the sum over $j$ is really the double sum over $(j_r, j_c)$, where the column terms are constant with respect to $j_r$. Each update is $2N$ one-dimensional $\LSE$s, so an iteration costs $O(N(h+w))$ FLOPs; on a square grid ($h = w = n_g$) this is $O(n_g^3) = O(N^{3/2})$, a factor $\sqrt N$ below the Point Cloud and General solvers, and the source of the Grid solver's wall-clock speed.

\paragraph{Initialization and termination.} In the Grid solver, each initialization and termination check operation costs no more than $O(N(h+w))$ FLOPs. \textsc{Normalize} can be performed in $O(h + w)$, and in $O(1)$ for nondecreasing axis costs. Building the axis tables $\ell_{\mathrm{r}}, \ell_{\mathrm{c}}$ is $O(h+w)$; the marginal flooring is identical to that of the previous two formats, in $O(N)$. The initial iterates $(X^0, v^0)$ map to $\xi = \zeta = -(\log N) \mathbf{1}_N$ and $\log v = \mathbf{0}_N$, as seen in the Point Cloud solver, in $O(N)$. \textsc{PrimalRes} evaluates the row marginals by reusing~\eqref{eq:sep-lse}. In particular, by~\eqref{eq:factored-plan},
\[
\sum_j X^k_{ij} = \exp\big(\xi^k_i + \LSE_j(kL_{ij} + \zeta^k_j)\big),
\]
which is exactly~\eqref{eq:sep-lse} with $\tilde \beta$ replaced by $\zeta$. \textsc{CTransform} factors along the two axes exactly as~\eqref{eq:sep-lse} does. The c-transform $\beta'_j=\min_i(C_{ij}-\alpha_i)$ over the source points $i = (i_r, i_c) \in [h] \times [w]$ becomes the double minimum over $(i_r, i_c)$, in which the column term $\varphi_{\mathrm{c}}(|i_c-j_c|)/C_{\max}$ is constant with respect to $i_r$ and moves outside the inner minimization:
\begin{equation*}
  \beta'_{j_r,j_c}=\min_{i_c}\Big(\varphi_{\mathrm{c}}(|i_c-j_c|)/C_{\max}+\min_{i_r}\big(\varphi_{\mathrm{r}}(|i_r-j_r|)/C_{\max}-\alpha_{i_r,i_c}\big)\Big).
\end{equation*}
This is two nested minimizations over the two axes, computable in $O(N(h+w))$ FLOPs. A similar argument can be made for the c-transform in the other direction, to obtain $\alpha'$. The primal objective, the dominant term of \textsc{Gap}, factors along the two axes as well:
\begin{align*}
  \langle C,X^k\rangle &=\frac{1}{C_{\max}}\sum_{ij} \varphi_{\mathrm{r}}(|i_r-j_r|)X^k_{ij}+\frac{1}{C_{\max}}\sum_{ij} \varphi_{\mathrm{c}}(|i_c-j_c|)X^k_{ij} \\
  &= \frac{1}{C_{\max}}\sum_{i_r,j_r} \varphi_{\mathrm{r}}(|i_r-j_r|) \underbrace{\sum_{i_c, j_c} X_{ij}^k}_{=:R^{\mathrm{r}}_{i_r,j_r}}+\frac{1}{C_{\max}}\sum_{i_c,j_c} \varphi_{\mathrm{c}}(|i_c-j_c|)\underbrace{\sum_{i_r, j_r} X_{ij}^k}_{=:R^{\mathrm{c}}_{i_c, j_c}},
\end{align*}
since each cost $\varphi$ is constant with respect to the axis along which we are aggregating. Substituting the factored plan~\eqref{eq:factored-plan} with $L_{ij}=\ell_{\mathrm{r}}(|i_r-j_r|)+\ell_{\mathrm{c}}(|i_c-j_c|)$,
\begin{align*}
R^{\mathrm{r}}_{i_r,j_r} &= \sum_{i_c}\sum_{j_c} \exp\big(\xi_{i_r,i_c} + k\ell_{\mathrm{r}}(|i_r-j_r|) + k\ell_{\mathrm{c}}(|i_c-j_c|) + \zeta_{j_r,j_c}\big) \\
&= \exp\Big(k\ell_{\mathrm{r}}(|i_r-j_r|) + \LSE_{i_c}\big(\xi_{i_r,i_c} + \LSE_{j_c}\big(k\ell_{\mathrm{c}}(|i_c-j_c|)+\zeta_{j_r,j_c}\big)\big)\Big).
\end{align*}
A similar scheme can be derived for $R^{\mathrm{c}}$. In total, the primal objective costs three new axis-length passes, in $O(N(h+w))$ FLOPs. A check that passes the proxy therefore costs nine axis-length passes, which is a little over two iterations, and one that fails the proxy costs the same as about half of an iteration.

\paragraph{GPU implementation.} Algorithm~\ref{alg:grid} is a description of our Grid solver, which implements the template of Section~\ref{sec:c-transform} for the Grid format. Our kernels specialize to the square grids of the benchmark dataset, $h = w = n_g$ with $n_g$ a power of two between $4$ and $1024$ and $\varphi_{\mathrm{r}} = \varphi_{\mathrm{c}} \in \{\delta, \delta^2\}$, collapsing the two axis tables to a single table $\ell$; for $h \ne w$ or $\varphi_{\mathrm{r}} \ne \varphi_{\mathrm{c}}$ the released code contains only a PyTorch reference of the iteration, without the termination checks of Algorithm~\ref{alg:grid}. Each update is one fused kernel with a thread block per grid row. For the $\xi$-update, the thread block of row $i_r$ owns the $i_r$-th row of the intermediate table $\log S$, which is computed once and reused by all $n_g$ outputs of the row. The block stages the table $k\ell$ and this intermediate row in shared memory, allowing fast access during both $\LSE$ passes. The $\zeta$-update kernel is implemented similarly. Within the block, each of the $n_g$ outputs is computed by $W$ threads that cooperate on its $\LSE$ operation ($W = 32, 16, 8$ at $n_g = 32, 64, 128$, so that $n_g W = 1024$ fills the block; from $n_g=256$ onwards, we set $W=1$ and the block is $n_g$ threads), each thread updating its own private online-$\LSE$ state~\eqref{eq:online-lse}. Finally, the per-thread states are merged by~\eqref{eq:lse-merge} in shared memory. An iteration launches three kernels, the $\xi$-update, the $\zeta$-update, and a counter advance that also rewrites the table $k\ell$, which are all captured once per solve as a CUDA graph and replayed.

\begin{rema}\label{rem:grid-roofline}
In the accounting of Section~\ref{sec:impl-factored}, an update of the Grid solver evaluates $2n_g^3$ summands. The intermediate table $\log S$ has $N = n_g^2$ entries, each an $\LSE$ over $n_g$ terms, and the $N$ outputs are each an $\LSE$ over $n_g$ terms. Each summand costs one addition to combine the table entry $k\ell$ with the vector or intermediate term, and three more for the $\LSE$ update~\eqref{eq:online-lse} for a total of four operations. That is $8n_g^3$ operations per update, and in total, $16n_g^3$ FLOPs per iteration. On the memory side, the $\xi$-update reads $\tilde\beta$, $\log\tilde r$ and the old $\xi$ and writes $\xi$ and $\log u$; the $\zeta$-update reads $\xi$, $\log\tilde c$ and the old $\zeta$ and writes $\zeta$, $\log v$ and the next $\tilde\beta$. This amounts to eleven passes over $N = n_g^2$ elements, or $88n_g^2$ bytes per iteration, counting each vector once. In the first pass of each update, every one of the $n_g$ thread blocks reads the whole input vector ($\tilde\beta$ in the $\xi$-update, $\xi$ in the $\zeta$-update), which adds loads of $16n_g^3$ bytes per iteration. One such vector occupies $8n_g^2$ bytes, which fits in the $40$~MB L2 cache for $n_g \le 2048$, and we assume that these repeated reads are served from L2, so that the HBM traffic is the eleven passes. The arithmetic intensity is therefore at least $16n_g^3/(88n_g^2) \approx 0.2n_g$ FLOP per byte, above the $4.8$ ridge from $n_g = 32$ onwards.

At smaller scales, the absolute scale of the arithmetic is small enough that its runtime on our device is comparable to that of a kernel launch. Concretely, at $n_g = 64$, an iteration costs about $4 \times 10^{6}$ FLOPs, which is under a microsecond at the device's peak rate~\cite{nvidia2021a100datasheet} and at the scale of a kernel launch. So, up to $n_g \approx 64$, the iteration cannot fill the device and launch latency governs. This makes the CUDA graph an essential part of the Grid solver for small-scale instances. For larger instances, the model predicts that the compute time overshadows the launch times, so that the kernels are no longer bound by launch latency or by HBM bandwidth.
\end{rema}

\section{Numerical Experiments}\label{sec:experiments}

In this section, we conduct numerical experiments to
\begin{enumerate}
  \item quantify the difference in computational speed between the three BDRS solvers, across the three formats;
  \item compare the three BDRS solvers against the GPU-accelerated baselines in their respective formats;
  \item quantify how much the c-transform post-processing reduces the number of iterations required to satisfy the stopping criterion for the Douglas--Rachford family of solvers, including ours and two baselines, and
  \item measure how the iteration count of BDRS responds when the step size $\eta$ is perturbed around its deployed value.
\end{enumerate}

The software versions used in this section, including the releases of POT and OTT-JAX and the upstream commit of each other baseline, are listed in the README of our code release.

\subsection{Experiment setup} \label{sec:experiment-setup}
\paragraph{Datasets.} Three datasets are used in this section, one for each DOT format. Note that for a dataset of DOT instances with more structure (e.g. Grid), we can \emph{lift} it into its less structured counterpart by writing out the cost function $\kappa$ (Point Cloud) or computing and storing the explicit $C$ (General).

\textbf{DOTmark}~\cite{schrieber2016dotmark} natively supplies Grid and Point Cloud instances, some of which are further lifted to General instances. It contains ten distinct classes of images, each class holding ten grayscale square images. The dataset supplies these images at various resolutions $n_g \in \{32, 64, 128, 256, 512\}$, so that each image can be represented as an $n_g \times n_g$ matrix of nonnegative values. Similar to~\cite{xia2026halo}, we derive an $n_g = 1024$ tier by corner-anchored bilinear upsampling of the supplied $n_g = 512$ images, rounded to integers before normalization. Next, within a class, a Grid DOT instance can be constructed for each pair of images of equal resolution, resulting in $\binom{10}{2}=45$ distinct DOT instances for each class and $450$ for each resolution and cost. The benchmark of Sections~\ref{sec:exp-formats} and~\ref{sec:exp-baselines} uses all of them; the studies of Section~\ref{sec:exp-ablation} use fixed subsets of five pairs per class (three for the $T_{\mathrm{chk}}$ sweep). To construct the DOTmark instances in Grid or Point Cloud format, we first normalize the pixel values of both images to a probability distribution over the $n_g \times n_g$ grid, so that $N = n_g^2$. Then, to obtain Grid (resp. Point Cloud) instances, we select either the $\ell_1$ or $\ell_2^2$ (resp. the $\ell_\infty$ or the unsquared $\ell_2$) cost as the cost between grid coordinates. We lift the Grid DOTmark instances into Point Cloud DOTmark instances, and then lift all Point Cloud DOTmark instances into General DOTmark instances, so that benchmarking experiments can be run on any (cost, format) pair.

\textbf{ModelNet10}~\cite{wu2015shapenets} natively supplies the Point Cloud instances. It is a library of computer-aided-design (CAD) 3-dimensional models of common indoor objects, distributed as triangle meshes and grouped into ten categories. From each category we take the first three shapes of its test split and pair them within the category, so that the two clouds of an instance are always distinct objects of the same kind. This gives $\binom{3}{2} \cdot 10 = 30$ pairs of objects across the dataset. Each CAD mesh is sampled area-weighted into a surface point cloud, which is then scaled isotropically into $[0,1]^3$. Every instance puts equal mass on its points, giving the uniform distribution for both marginals $r = c = \tfrac{1}{N}\mathbf{1}_N$. For each object, $N_{\max}=2^{19}$ samples are drawn \emph{once} throughout our experiments. Then, we randomly fix an ordering of these samples so that taking the first $N \le N_{\max}$ points always gives a nested, reproducible subsample that is individually unbiased. Moreover, we select the squared Euclidean distance as the cost between the two point clouds.

Lastly, we use \textbf{random dense costs} as a way to generate natively General format DOT instances. In particular, each entry of the cost matrix $C_{ij}$ is sampled independently from either the uniform distribution on $[0,1]$ or the half-normal distribution $|\mathcal{N}(0,1)|$, with marginals drawn from the flat Dirichlet distribution $\mathrm{Dir}(1, \dots, 1)$. Cost matrices drawn entrywise at random are standard in both studies regarding DOT problems~\cite{schwinn2019heuristics} and related assignment problems~\cite{dellamico2000assignment}. In our experiments, we use this procedure to generate ten instances, each with its own pair of marginals, for each size $N \in \{2^{10}, 2^{12}, 2^{14}, 2^{16}\}$ and each of the two cost families, eighty in total. Cost matrices generated under this scheme have no underlying geometric structure, so no Point Cloud or Grid method applies to them, and they are only presented to demonstrate the efficacy of General format solvers.

\paragraph{Protocol.} Every solve receives one A100-SXM4-80GB with its full $80$~GB of device memory, $16$ CPU cores with $93$~GB of host memory, and a wall-clock limit of $1800$~seconds; a solve that exceeds any of the three is recorded as such and considered a failure. Every solver, including ours and baselines, runs to its \emph{own native} stopping rule at its default settings, except for the following:
\begin{enumerate}
    \item DROT, PDOT and cuPDLPx receive the iteration cap of $10^6$,
    \item for fairness, our driver of DROT tests its residual every $100$ iterations rather than at every iteration and, since DROT's solver routine sets no default tolerance or step size, uses the tolerance $10^{-4}$ and step size $\sigma = 2/(m+n)$ that the implementation's example program passes, and
    \item OTT's convolutional Sinkhorn runs in its scaling mode rather than its default log-domain mode (Appendix~\ref{app:duals}).
\end{enumerate}
Before the timed solve, every solver except DROT runs an untimed warm-up on a dummy instance that shares no data with the benchmark instance, so that compilation and first-call allocation are not timed. DROT shipped no warm-up routine. Faithful to its original implementation, PDOT additionally runs an untimed $100$-iteration warm-up on the instance itself before each measured solve; the result of this warm-up is discarded, and the measured solve starts from a freshly initialized solver state. From each solver's delivered state, we recover a representation of the primal--dual triple $(X,\alpha,\beta)$ for the original problem~\eqref{eq:primal-problem} and score it by the KKT residual~\eqref{eq:kkt}. Note that not every baseline implementation returns a pair of dual potentials $(\alpha, \beta)$, so we derived mappings from each baseline's returned state to its best estimate of the dual variables, i.e. solutions to~\eqref{eq:dual-problem}, in Appendix~\ref{app:duals}. The comparison is therefore one of the recovered primal--dual pairs under these recovery rules, not of the primal plans alone. Furthermore, we measure two timings for each solve: $t_{\mathrm{solve}}$ times the solver call alone and stops when the solver returns at its native stopping criterion; $t_{\mathrm{cert}}$ adds to it the post-solve recovery of $(X, \alpha, \beta)$, which could involve the retrieval of the transport plan (HOT) or the c-transform that gives LAMP's $\alpha$, and a potentially expensive computation of the dual residual (HALO, DROT, M3S, and HOT). For our solvers, which terminate with a certificate of their own, we always have $t_{\mathrm{solve}}=t_{\mathrm{cert}}$. In Subsections~\ref{sec:exp-formats} and~\ref{sec:exp-baselines}, we present the comparison across solvers as an accuracy--time Pareto chart over the \emph{co-solved} set of DOT instances to ensure that no solver is evaluated on a harder subset of problems than another. The co-solved set of a cell (one choice of dataset, size and cost) holds the instances that every solver drawn in the figure has solved, and the point of each tolerance aggregates that tolerance's solves within the set. A solve counts as solved (under the solver's native stop), or \emph{scored}, when it ends within the limits of this protocol with finite $r_P$, $r_D$ and $g$, whatever its KKT residual. A point that aggregates fewer instances than its cell's co-solved set is drawn hollow, and a dominance statement involving such a point is checked on the instances that both solvers return. Table~\ref{tab:outcomes} in Appendix~\ref{app:outcomes} gives, for every solver and DOTmark resolution, the share of attempted solves that are scored and the share that reach a KKT residual of at most $10^{-4}$. Each solver contributes a single pair $(t_\mathrm{cert}, \KKT)$, with its $t_{\mathrm{solve}}$ marked by a thin line with a tick, when applicable. We occasionally score our solver under various values of input tolerances to map out a \emph{Pareto frontier}. Each point aggregates its $s$ scored solves, one per instance, by the shifted geometric mean $\mathrm{SGM10} := \exp\bigl(\frac{1}{s}\sum_{i=1}^{s} \log(t_i + 10)\bigr) - 10$ of their times $t_i$ in seconds and by the geometric mean (GM) of their KKT residuals, all of which are positive and finite; a solve that is not scored enters neither mean. A solve is killed if it has not finished $30$ seconds after the time limit, a margin for its start-up, warm-up and scoring; a killed solve, and any solve whose $t_{\mathrm{solve}}$ exceeds the time limit, is counted as a timeout. PDOT and cuPDLPx also receive the time limit as a parameter, and a solve that ends on it is likewise counted as a timeout. LAMP runs its own stopping test with its tolerance set to the benchmark's $10^{-4}$ and its native time limit of $600$ seconds; a solve that ends on that limit is LAMP's own exit, not a failure, and is scored like any other.

\paragraph{DOT instances.} DOT instances $(C, r, c)$ across all three formats are presented in normalized form, where $\max (C)= 1$, and every solver of the same format receives the same data, except for the following:
\begin{enumerate}
    \item M3S and the grid mode of HALO build the squared Euclidean cost of the grid themselves, at their own scales, so we rescale their objectives and potentials to the normalized cost before scoring, and
    \item DROT and HALO receive their inputs rounded to single precision, while every other solver runs in double precision. Every delivered state is scored in double precision.
\end{enumerate}
Our solvers' \textsc{Normalize} steps are no-ops on benchmark input, except where rounding leaves $\max(C) = 1 - 2^{-53}$ (the unsquared $\ell_2$ cost at $n_g \le 64$ and two half-normal instances of Fig.~\ref{fig:exp21p}), but are still performed in every solve to simulate a real use case. The marginal flooring procedure is internal to our three solvers and never applied to the marginals any KKT residual is evaluated against. In addition, the marginals of M3S are floored at $10^{-12}$ because it aborts on zero entries, and those of LAMP, through its own rule $r \mapsto (r + 10^{-6}\mathbf{1}_m)/\|r + 10^{-6}\mathbf{1}_m\|_1$.

\paragraph{Configuration of our solvers.} All three of our solvers take two hyperparameters, the step size $\eta$ and the check period $T_{\mathrm{chk}}$. In our benchmarks, both are set by a fixed rule computed before the solve and held constant throughout it. The step size is determined by the rule $\eta = 0.7 \times \overline{C}$, where $\overline{C}$ denotes the mean entry of the cost matrix. The check period is a table on problem size, with $T_{\mathrm{chk}} = 100$, $200$, $500$, $1000$, $2000$ and $4000$ at $n_g = 32$, $64$, $128$, $256$, $512$ and $1024$, respectively. For every other instance, $T_{\mathrm{chk}}$ is read from the same table at the equivalent resolution $n_g = 2^{\lceil \log_2(N)/2 \rceil}$, clamped to the table's range. Every benchmark solve uses the iteration budget $K_{\max} = 10^{6}$.

\paragraph{Baselines by format, and scaling limits.} As discussed in Section~\ref{sec:intro}, there is a wide spectrum of GPU-accelerated solvers for the DOT problem. Table~\ref{tab:ceilings} lists the nine solvers we have selected, as well as their respective DOT format, to be compared to the corresponding implementation of BDRS. LAMP runs in two formats, General and Point Cloud, but is counted as one baseline. Further, we provide the largest DOTmark resolution at which each listed solver, under the stated protocol and its native stopping rule, returns a solution with a finite KKT residual on more than $75\%$ of the instances it attempts, whatever that residual (for our solvers at $\mathrm{tol} = 10^{-4}$, and for the convolutional Sinkhorn at its best $\varepsilon$). DOTmark resolution is the natural ruler for this because the dataset is Grid-native and lifts to all three formats, so that we can put every solver on the same scale. We concede that DOTmark steps are coarse (a $16\times$ increase in the number of decision variables per step), so the table gives rough estimates of scaling boundaries only.

\begin{table}[t]
\caption{The list of GPU-accelerated DOT solvers considered and their associated format, as well as the resolution of the largest DOTmark they can solve and why they cannot solve larger ones. The rows of our three solvers are for the stopping tolerance $\mathrm{tol} = 10^{-4}$, and the $n_g = 1024$ tier is upsampled from the $n_g = 512$ images. \emph{Full ladder}: every size was solved. \emph{Timeout}: the next size hit the $1800$-second limit on all or most instances attempted. \emph{Device OOM} and \emph{host OOM}: the next size exceeds the device or the host memory allocated to a solve, measured or projected before launch. \emph{Solver abort}: the solver aborts its solve on most instances at the next size. \emph{Kernel fault}: the solver's upstream kernels fault (illegal memory access) at the next size, so it was not run there. \emph{Native time limit}: every solve at this size already ends on the solver's own time limit, so the next size was not run. $^{*}$M3S requires its supports to be a grid but does not require a separable cost, so we classify it as a Point Cloud format solver for the purposes of our comparisons. At $n_g = 512$ it returns a finite solution on $190$ of the $450$ instances and aborts on the other $260$, and its published code does not accept $n_g = 1024$.}
\label{tab:ceilings}
\centering
\begin{tabular}{llll}
\hline\noalign{\smallskip}
Solver & Format & Largest $n_g$ solved & Why not larger \\
\noalign{\smallskip}\hline\noalign{\smallskip}
Grid BDRS (ours) & Grid & $1024$ & full ladder \\
HOT & Grid & $512$ & device OOM \\
Conv.\ Sinkhorn & Grid & $1024$ & full ladder \\
Point Cloud BDRS (ours) & Point Cloud & $256$ & timeout \\
HALO & Point Cloud & $512$ & timeout \\
M3S & Point Cloud$^{*}$ & $256$ & solver abort \\
LAMP & Point Cloud & $256$ & native time limit \\
General BDRS (ours) & General & $256$ & device OOM \\
DROT & General & $128$ & kernel fault \\
PDOT & General & $128$ & device OOM \\
cuPDLPx & General & $64$ & host OOM \\
Sinkhorn & General & $128$ & device OOM \\
LAMP & General & $128$ & device OOM \\
\noalign{\smallskip}\hline
\end{tabular}
\end{table}

\subsection{The cost of generality across the three formats}\label{sec:exp-formats}
A natural first question to ask would be: \emph{does the computational cost of DOT across the three formats differ significantly?} In this subsection, we answer this question affirmatively, and find that the results double as an overview of the time--accuracy performance of all the DOT solvers under consideration.

\begin{figure}[t]
\centering
\includegraphics[width=0.99\textwidth]{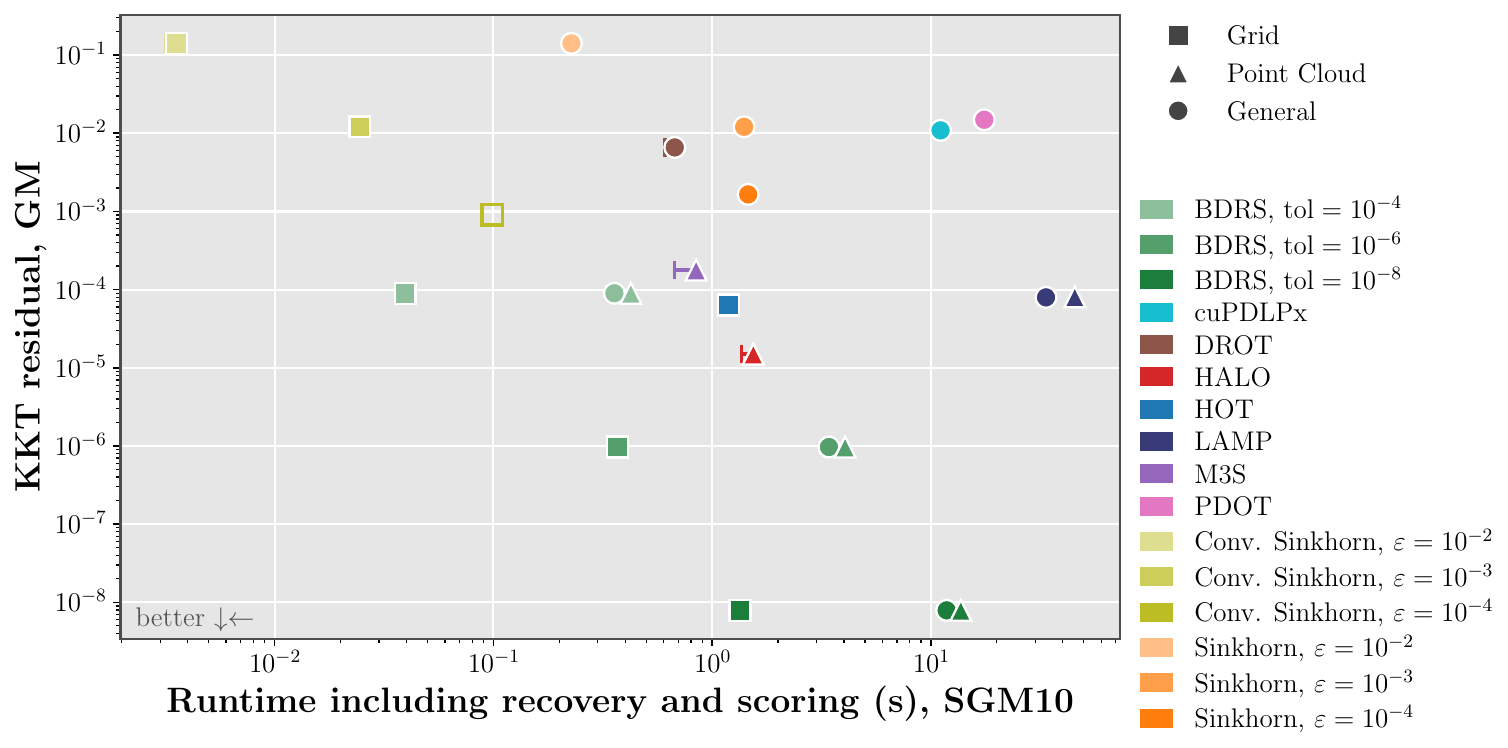}
\caption{Comparison of all solvers across all three formats on the same $414$ co-solved DOTmark instances at resolution $n_g=64$ with $\ell_2^2$ cost, those of the $450$ on which PDOT finished inside the time limit. We use a different color for each baseline and one color for our three solvers. Shades represent the tolerance for our solvers and regularization $\varepsilon$ for the Sinkhorn variants. Shapes indicate the format that the implementation consumes. HALO runs its grid mode on these instances. The $\varepsilon = 10^{-4}$ convolutional Sinkhorn returns non-finite potentials on $148$ of the $450$ instances, so it does not take part in forming the co-solved set; its point, drawn hollow, aggregates the $286$ co-solved instances on which its solution is finite.}
\label{fig:exp1}
\end{figure}

Figure~\ref{fig:exp1} presents the Pareto plot of the solvers over the DOTmark instances at $n_g=64$, with the $\ell_2^2$ ground cost. We lift these instances to both Point Cloud and General, so that we can pass this set of instances to all the solvers listed in Table~\ref{tab:ceilings} in their desired format. Each of our three solvers was run three times, at $\mathrm{tol} = 10^{-4}, 10^{-6}, 10^{-8}$, respectively. We observe that at each tolerance, all three of our BDRS solvers terminate at the same accuracy that is under the given tolerance, which is a validation of the fact that they run identical iterations under mathematically equivalent representations. Moreover, Grid BDRS achieves a $t_{\mathrm{cert}}$ of $0.040$~s, which is $11\times$ and $9\times$ faster than Point Cloud BDRS and General BDRS, respectively, demonstrating significant computational gaps even at the scale of $n_g=64$.

Moreover, the Point Cloud solver consumes more structure than General and is nonetheless the slower of the two, by $1.19\times$ at this cell. The reason is the memory--recompute tradeoff between the two. The General solver needs to store an $m \times n$ cost matrix to read from it, while the Point Cloud solver recomputes it from coordinates every time it is needed so that the device need not be able to hold an $m \times n$ array at scale. For a practitioner the rule is therefore: use the Grid format when the cost is separable; otherwise materialize the cost and use the General solver whenever it fits, since it is faster from $n_g = 64$ upward. Lastly, use the Point Cloud format when it does not fit, which on an $80$~GB device begins near $N = 2^{16}$.

Zooming out, every Grid baseline solver except the convolutional Sinkhorn at $\varepsilon = 10^{-2}$ and $10^{-3}$ is Pareto-dominated by our Grid solver at some choice of tolerance. Those two are $11\times$ and $1.6\times$ faster than our Grid solver at $\mathrm{tol} = 10^{-4}$ in SGM10 ($0.0036$ and $0.025$ against $0.040$~s), but their KKT residuals are $1.6 \times 10^{3}$ and $1.3 \times 10^{2}$ times larger, respectively. Our General solver Pareto-dominates every General solver except Sinkhorn at $\varepsilon = 10^{-2}$, where our General solver at $\mathrm{tol}=10^{-4}$ achieves a geometric mean of KKT residual $1.6 \times 10^{3}$ times smaller than the Sinkhorn solution, at the cost of $1.6\times$ runtime in SGM10 ($0.36$ against $0.23$~s). Lastly, our Point Cloud solver at $\mathrm{tol}=10^{-4}$ Pareto-dominates M3S while achieving a $3.6\times$ speedup over HALO at the cost of being $6.0\times$ less accurate in terms of the geometric mean of KKT residual. LAMP, in both its General and its Point Cloud form, is Pareto-dominated by our solver of the same format at $\mathrm{tol} = 10^{-6}$, which reaches a $\KKT$ of $9.7 \times 10^{-7}$ in $3.4$ and $4.1$~s against LAMP's $8.0$ and $8.1 \times 10^{-5}$ in $34$ and $45$~s. In summary, for this experiment, our solvers provide a favorable accuracy--time tradeoff against the corresponding baseline solvers in every format.

\subsection{Format-by-format comparison against the GPU baselines}\label{sec:exp-baselines}
Although Figure~\ref{fig:exp1} provides a head-to-head comparison for every solver in Table~\ref{tab:ceilings}, it does not test the scalability and the generalizability of the solvers. This subsection therefore performs in-depth comparisons between our implementations and the baselines on multiple datasets of varying sizes, in their respective format.

\subsubsection{The General format}
\textbf{Randomly generated cost} and \textbf{lifted DOTmark} are used to benchmark the performance of solvers that take in General DOT instances (which we call ``General solvers'' for the remainder of this section).

\begin{figure}[tp]
\centering
\includegraphics[width=\textwidth]{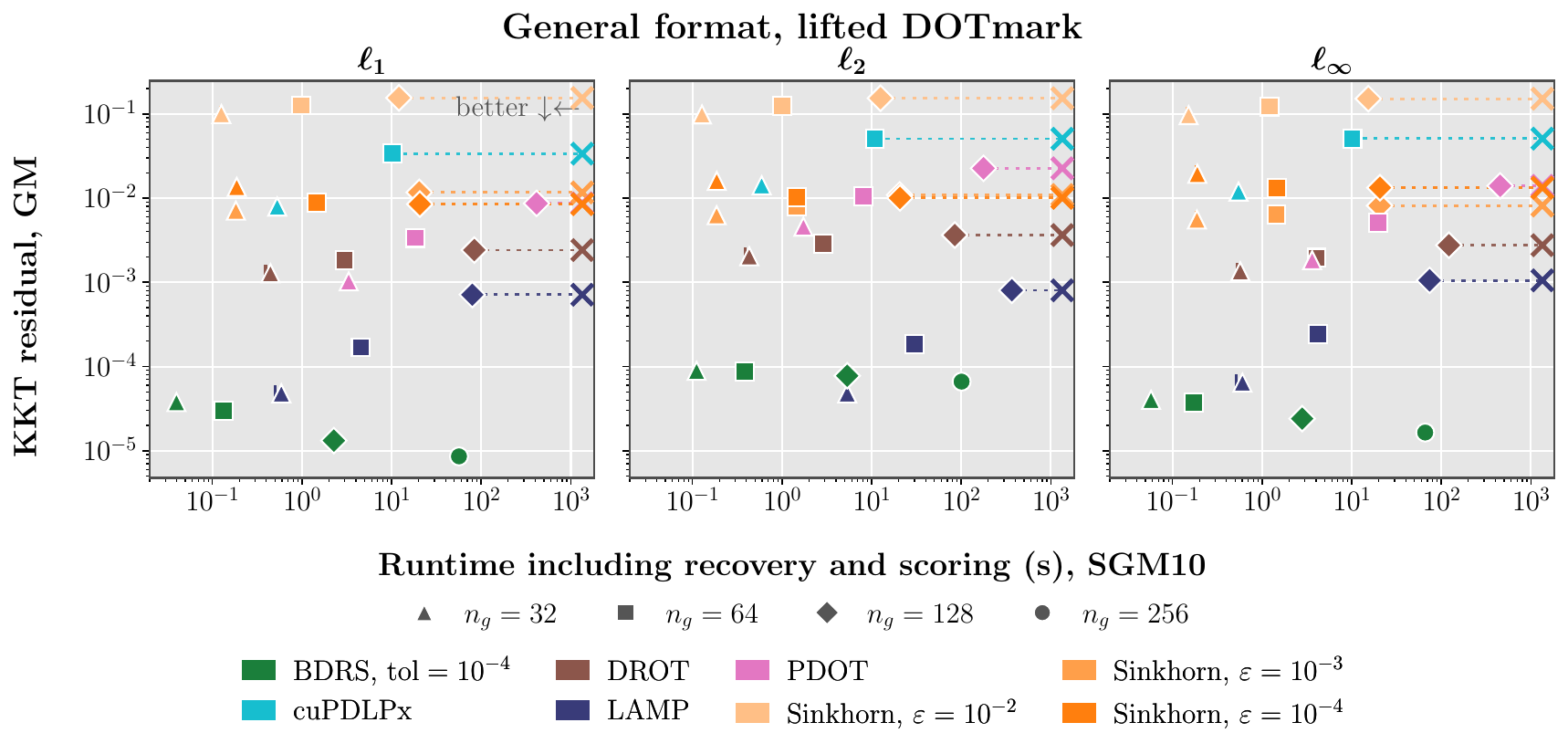}
\caption{Our General BDRS solver against the baseline General solvers on lifted DOTmark. In particular, the $\ell_1, \ell_2,$ and $\ell_\infty$ costs were used at each size $n_g \in \{32, 64, 128, 256\}$. Every point aggregates all $450$ instances of its choice of size and cost on DOTmark, except at $n_g = 128$ with the $\ell_1$ and $\ell_\infty$ costs, where the co-solved sets hold the $447$ and $441$ instances on which PDOT finished inside the time limit. Moreover, at $n_g=256$, our General solver solves all instances across all costs to $\mathrm{tol}=10^{-4}$. The crosses on each panel mark cuPDLPx's failed attempt at $n_g=128$ due to a host-side OOM and, at $n_g=256$, the solves that DROT, PDOT, LAMP and the three Sinkhorn arms decline. DROT runs into a kernel fault at $n_g=256$.}
\label{fig:exp21}
\end{figure}

Figure~\ref{fig:exp21} compares the General solvers on \textbf{lifted DOTmark} at $n_g \in \{32, 64, 128, 256\}$, one panel per cost. Across the three, our solver is the fastest solver, among those that solve~\eqref{eq:primal-problem} and those that solve its entropic approximation alike, at every size and cost, and its margin over DROT and PDOT widens with problem size. The margin is the narrowest at $n_g = 32$ on the unsquared Euclidean cost, where DROT is $3.9\times$ slower ($0.43$ against $0.11$~s) and the $\varepsilon = 10^{-2}$ Sinkhorn arm $1.1\times$ ($0.13$~s); at $n_g = 128$ ours leads DROT by $16\times$ to $43\times$ ($84$ to $121$~s against $2.3$ to $5.3$~s), PDOT by $33\times$ to $183\times$ ($177$ to $453$~s), and the three dense Sinkhorn arms by $2.3\times$ to $9.1\times$ ($12$ to $21$~s). The widest speedup over DROT is on $\ell_\infty$, where ours takes $2.8$~s against $121$~s; over PDOT and the $\varepsilon \in \{10^{-3}, 10^{-4}\}$ Sinkhorn arms, it is on $\ell_1$, where ours takes $2.3$~s against $415$~s and $20$ to $21$~s. Our General solver is also the most accurate solver in every cell up to $n_g = 128$ except one, at a $\KKT$ of between $1.3$ and $8.9 \times 10^{-5}$, against $1.3$ to $3.6 \times 10^{-3}$ for DROT, $1.0 \times 10^{-3}$ to $2.3 \times 10^{-2}$ for PDOT, $4.8 \times 10^{-5}$ to $1.1 \times 10^{-3}$ for LAMP, and $5.6 \times 10^{-3}$ to $1.5 \times 10^{-1}$ for the three dense Sinkhorn arms. The exception is the unsquared Euclidean cost at $n_g = 32$, where LAMP reaches $4.8 \times 10^{-5}$ against our $8.9 \times 10^{-5}$ while taking $48\times$ longer ($5.34$ against $0.111$~s). As a result, apart from the aforementioned case, our General solver Pareto-dominates every other baseline solver, whether it solves~\eqref{eq:primal-problem} or its entropic approximation, at every problem size, for every choice of cost on the DOTmark dataset.

At $n_g = 256$ our General solver continues to perform consistently while the baseline solvers face memory restrictions or kernel faults. In particular, our General solver solves all $450$ instances on each of the three costs, at a $\KKT$ between $8.6 \times 10^{-6}$ and $6.6 \times 10^{-5}$, in $66$, $57$ and $101$~s on $\ell_\infty$, $\ell_1$ and the unsquared Euclidean cost, respectively. This observation is consistent with the memory bottleneck recorded in Table~\ref{tab:ceilings}.

\begin{figure}[t]
\centering
\includegraphics[width=\textwidth]{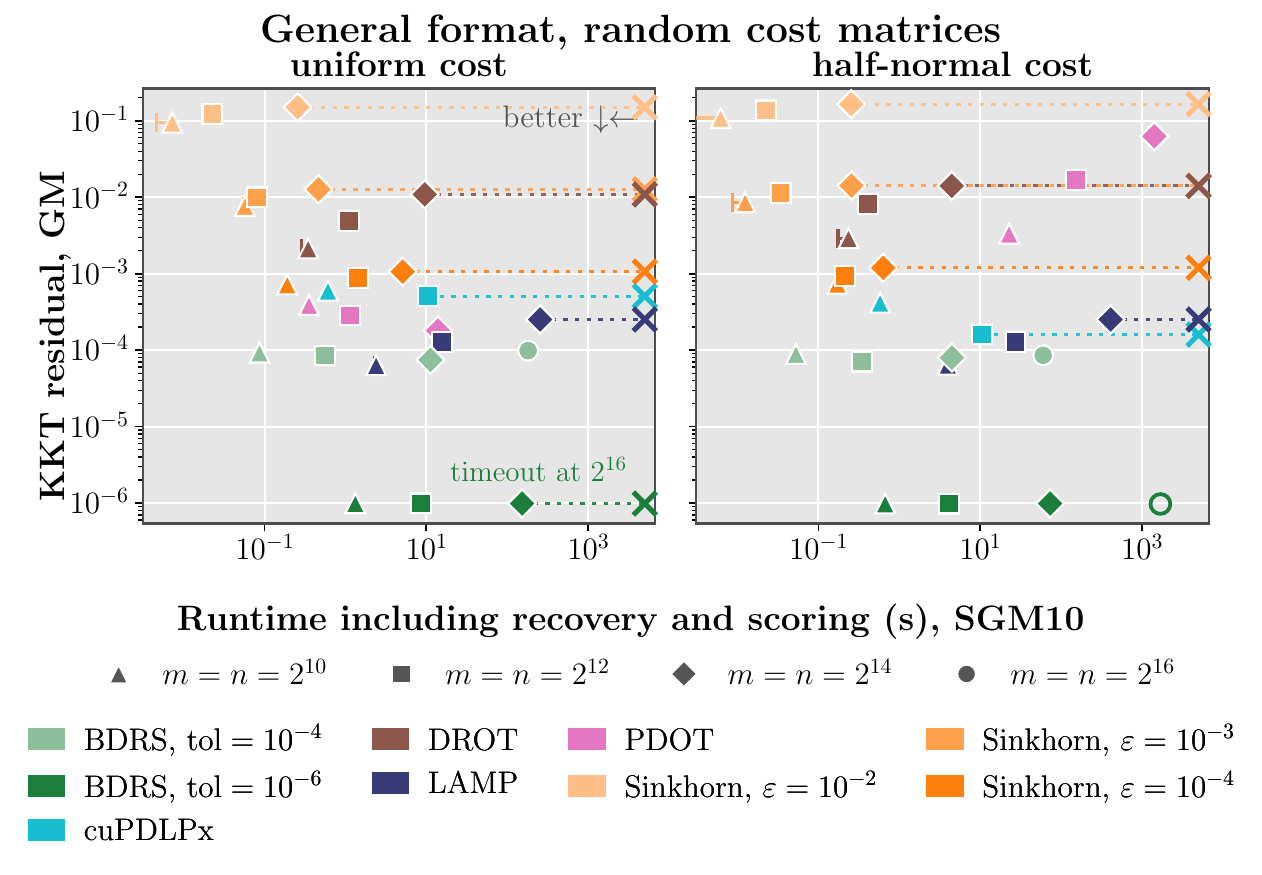}
\caption{Our General BDRS solver against the baseline General solvers on random cost matrices whose entries are drawn independently from the uniform distribution (left) and from the half-normal distribution (right), with ten instances per cost family at each size $m = n = N \in \{2^{10}, 2^{12}, 2^{14}, 2^{16}\}$, with our solver at $\mathrm{tol} \in \{10^{-4}, 10^{-6}\}$. A cross at the right edge marks the next size at which an arm has no scored solve: $N = 2^{14}$ for cuPDLPx, which runs out of host memory there, and $N = 2^{16}$ for DROT, LAMP and the three Sinkhorn arms, which decline that size before solving; PDOT is projected to time out and therefore is not run at $N = 2^{16}$. On the half-normal cost, the points at $N = 2^{12}$ and $2^{14}$ aggregate only the $9$ and $6$ of the $10$ instances on which PDOT finished inside the time limit. At $N = 2^{16}$, the half-normal points aggregate the $8$ of the $10$ instances on which our solver did not exceed device memory ($7$ at $\mathrm{tol} = 10^{-6}$, drawn hollow, where one more solve exceeds the time limit), and on the uniform cost, our $\mathrm{tol} = 10^{-6}$ arm exceeds the time limit on all $10$. Every solve of ours at $N = 2^{16}$ that completes peaks at $64.8$~GiB of device memory, holding $C$ and the dense plan $X$; on the two half-normal instances that exceed it, rounding leaves $\max(C) = 1 - 2^{-53}$, so \textsc{Normalize} also keeps $C/\max(C)$ as a second $2^{16} \times 2^{16}$ array, which no longer fits.}
\label{fig:exp21p}
\end{figure}

Figure~\ref{fig:exp21p} presents the accuracy--time Pareto plots of the General solvers on \textbf{random cost matrices}. Our General solver Pareto-dominates the baseline General solvers of~\eqref{eq:primal-problem} at $2^{10}$ and $2^{12}$ sizes for both the uniform cost and the half-normal cost. At $N = 2^{14}$ the picture changes for one baseline. cuPDLPx exceeds its host memory on both costs, and PDOT remains dominated on both panels. DROT, however, accepts a solution faster than our General solver by $15$ percent on the uniform cost ($9.6$ against $11.3$~s) and is essentially tied with it on the half-normal cost ($4.47$ against $4.49$~s), but at a $\KKT$ of $1.1$ and $1.4 \times 10^{-2}$, more than two orders of magnitude above ours. Although neither solver dominates the other, we claim that our solver remains the favorable trade at that size. LAMP is Pareto-dominated by our $\mathrm{tol} = 10^{-6}$ arm at $N = 2^{10}$ and by our $\mathrm{tol} = 10^{-4}$ arm at $2^{12}$ and $2^{14}$. It is projected to exceed device memory at $2^{16}$.

Among the entropic solvers, the $\varepsilon = 10^{-2}$ and $\varepsilon = 10^{-3}$ Sinkhorn implementations accept a solution faster than ours everywhere, but our General solver overtakes the $\varepsilon = 10^{-4}$ Sinkhorn variant at $N = 2^{10}$ on both costs and at $2^{12}$ on the uniform cost. Moreover, as expected, no Sinkhorn variant produces a primal--dual solution with a $\KKT$ less than $7.2 \times 10^{-4}$, and this residual grows by approximately one order of magnitude for every order of magnitude $\varepsilon$ increases. Meanwhile, our General solver achieves a $\KKT$ of less than $10^{-4}$, as determined by $\mathrm{tol}$, which is $7.8$ to $2070$ times lower than those of the Sinkhorn variants.

\subsubsection{The Point Cloud format}

\textbf{ModelNet10} and \textbf{lifted DOTmark} are used to benchmark the performance of the Point Cloud solvers, namely our BDRS Point Cloud solver, HALO's \texttt{tree} mode on ModelNet10 and its grid mode on lifted DOTmark, M3S, and LAMP's linear-space mode. Notably, M3S further requires the support of the marginal distributions to lie on a grid, but does not require the cost to be separable. Its published GPU solver evaluates only the squared Euclidean cost. This puts M3S's format specifications between the Grid and Point Cloud formats. For the purposes of this benchmark, we place it in the Point Cloud campaign. The same holds for HALO's grid mode, which also requires grid supports.

\begin{figure}[t]
\centering
\begin{subfigure}[t]{0.48\textwidth}
\centering
\includegraphics[width=\textwidth]{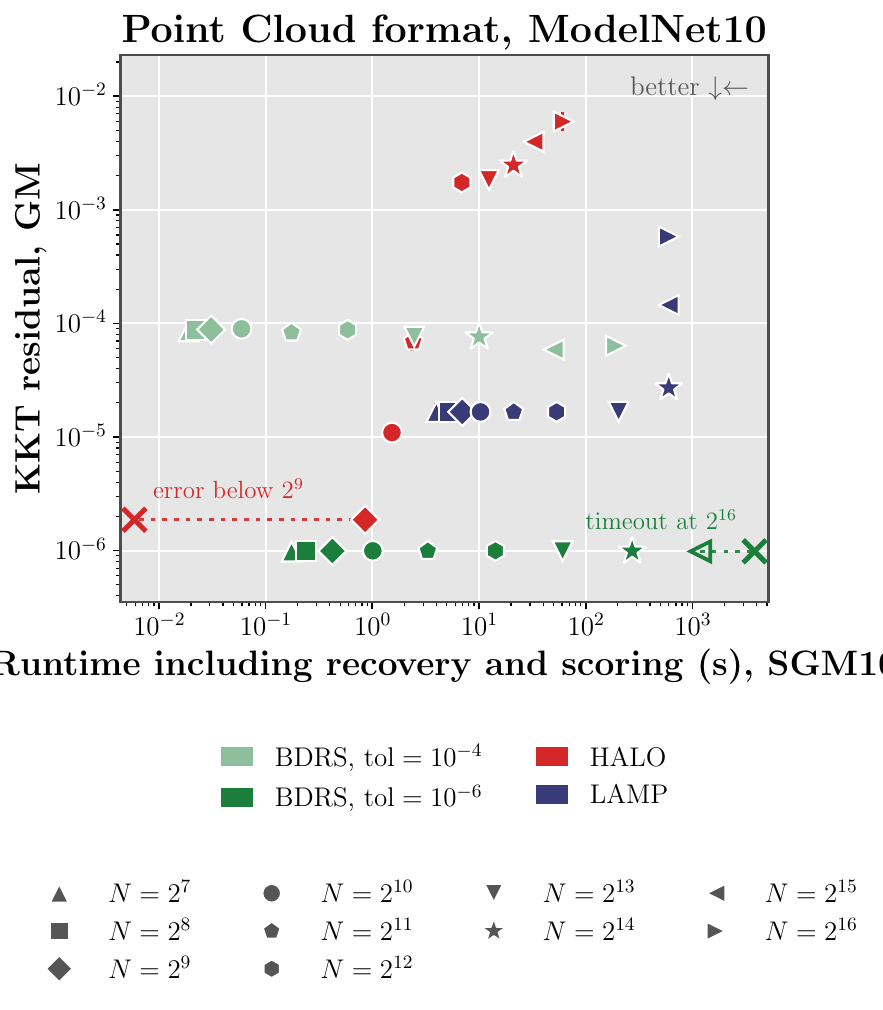}
\caption{ModelNet10.}
\label{fig:exp22}
\end{subfigure}\hfill
\begin{subfigure}[t]{0.48\textwidth}
\centering
\includegraphics[width=\textwidth]{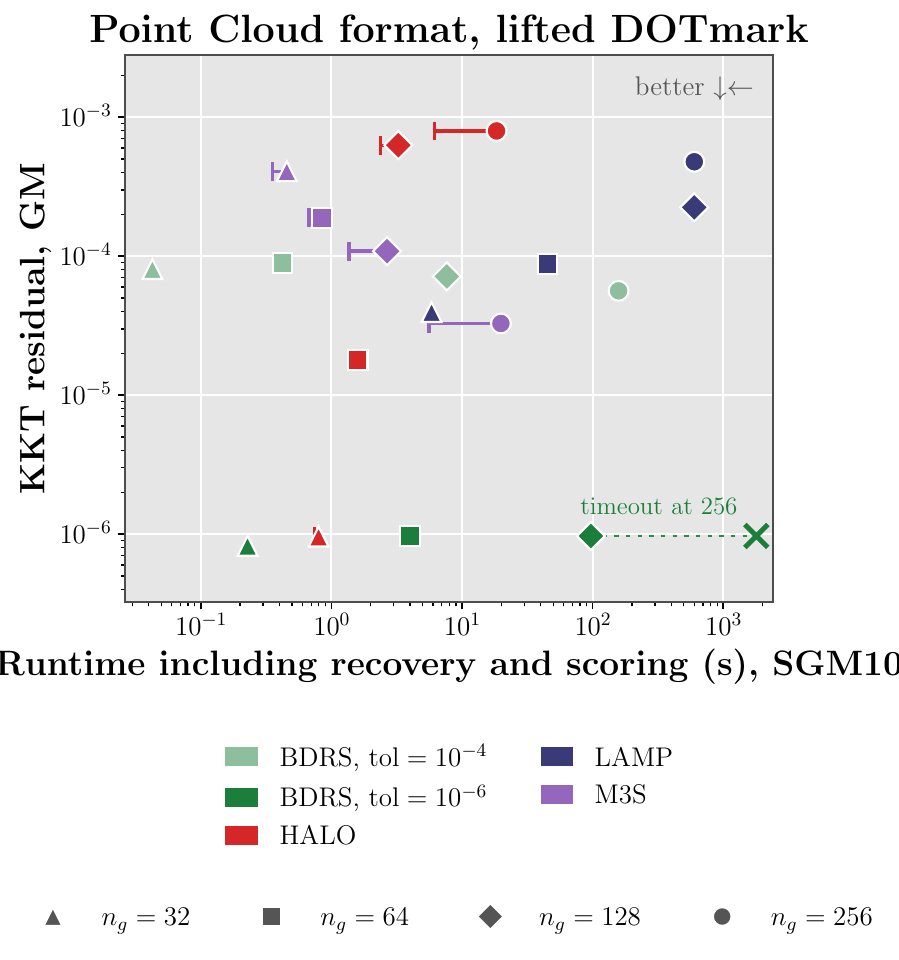}
\caption{Lifted DOTmark with the $\ell_2^2$ cost.}
\label{fig:exp22b}
\end{subfigure}
\caption{Our Point Cloud solver with $\mathrm{tol} \in \{10^{-4}, 10^{-6}\}$ against (a) HALO and LAMP on ModelNet10 and (b) HALO, M3S and LAMP on lifted DOTmark. (a) Our Point Cloud solver and HALO in its tree mode on ModelNet10 with the squared Euclidean cost, thirty pairs at each size from $N = 2^{7}$ to $2^{16}$ points per marginal. The cross at the left edge represents the fact that HALO rejects the $2^{7}$ and $2^{8}$ instances. (b) Our Point Cloud solver, HALO in its grid mode, M3S and LAMP on lifted DOTmark with the $\ell_2^2$ cost, with resolutions $n_g \in \{32, 64, 128, 256\}$. The co-solved set holds $450$ instances at $n_g = 32$ and $64$, $439$ at $128$ and $370$ at $256$. The excluded instances are those on which M3S aborted the solve due to ``\texttt{too many absorptions}'' at $n_g=128$ and $n_g=256$. Our $\mathrm{tol} = 10^{-6}$ arm reaches $N = 2^{15}$ in (a) and $n_g = 128$ in (b), the largest sizes at which it completes within the $1800$~s budget. One of the thirty instances at $N = 2^{15}$ exceeds it and is excluded, so that point is drawn hollow. The cross of this arm at $N = 2^{16}$ in (a) is a timeout projected from its time at $2^{15}$. The cross of this arm at $n_g = 256$ in (b) is a timeout. LAMP's points sit at its own $600$~s limit from $N = 2^{14}$ in (a) and at $n_g = 128$ and $256$ in (b).}
\label{fig:exp22pair}
\end{figure}
Figure~\ref{fig:exp22} compares our Point Cloud solver against HALO's \texttt{tree} mode on ModelNet10. HALO rejects instances with $N=2^7$ and $N=2^8$. Therefore, our solver and LAMP are the only ones that are able to solve ModelNet10 instances with sizes below $N=2^9$, where we do so successfully up to $\mathrm{tol}=10^{-6}$ for all instances. For $N \in \{2^9, \dots, 2^{14}\}$, we accept a solution faster, by $28\times$ at $2^{9}$ narrowing to $2.1\times$ at $2^{14}$, and beyond it HALO is faster, by $1.5\times$ at $2^{15}$ and $3.1\times$ at $2^{16}$. Our $\mathrm{tol} = 10^{-6}$ arm is more accurate than HALO at every size it reaches, and our $\mathrm{tol} = 10^{-4}$ arm is more accurate from $2^{12}$ on. Concretely, our KKT residual consistently stays below the desired tolerance, at $10^{-4}$ and $10^{-6}$, while HALO accepts solutions with a KKT residual that increases monotonically with $N$, ranging from $1.9 \times 10^{-6}$ at $N=2^{9}$ to $6.0 \times 10^{-3}$ at $N=2^{16}$. LAMP's linear-space mode is Pareto-dominated by our $\mathrm{tol} = 10^{-6}$ arm up to $N = 2^{14}$ and by our $\mathrm{tol} = 10^{-4}$ arm at $2^{15}$ and $2^{16}$.

Figure~\ref{fig:exp22b} runs the same comparison on lifted DOTmark with HALO in its grid mode and M3S added, and the pattern repeats. At $n_g = 32$ and $64$, our $\mathrm{tol}=10^{-4}$ arm accepts a solution $19\times$ and $3.8\times$ faster than HALO and $11\times$ and $2.0\times$ faster than M3S; at $n_g = 128$ and $256$ both baselines are faster than we are, HALO by $2.4\times$ and $8.6\times$ and M3S by $2.9\times$ and $8.0\times$. Our KKT residual is again flat across the four sizes, between $5.6 \times 10^{-5}$ and $9.0 \times 10^{-5}$, while HALO's drawn points rise from $9.6 \times 10^{-7}$ to $8.0 \times 10^{-4}$. M3S, which solves the entropic problem rather than~\eqref{eq:primal-problem}, is the least accurate of the three at the two smaller sizes, at $4.1 \times 10^{-4}$ and $1.9 \times 10^{-4}$. At $n_g=128$ and $n_g=256$, M3S aborts the solve for $11$ and $80$ instances out of the $450$ total instances, while HALO, LAMP and our Point Cloud solver solve all $450$. LAMP's linear-space mode is again Pareto-dominated by our $\mathrm{tol} = 10^{-6}$ arm at $n_g = 32$ and $64$, and by our $\mathrm{tol} = 10^{-4}$ arm at $128$ and $256$. Among the co-solved problems at $n_g=128$ and $n_g=256$, M3S demonstrates comparable-to-better performance both in terms of accuracy and runtime. In summary, our Point Cloud solver is competitive in runtime and accuracy on small instances, and generally trades slower solve time for better accuracy against HALO and M3S at larger scales. Among the baselines, HALO scales the best in runtime but rejects smaller instances in its tree mode and has degrading solution accuracy, while M3S achieves greater accuracy as problem size increases but becomes unreliable at medium-to-large scales.

\subsubsection{The Grid format}
\textbf{DOTmark} is used to benchmark the performance of solvers that take in Grid DOT instances (``Grid solvers''), at $n_g \in \{32, 64, 128, 256, 512, 1024\}$ (the $n_g = 1024$ tier upsampled from the $n_g = 512$ images) with the squared Euclidean and the $\ell_1$ cost. The Grid solvers consist of our Grid BDRS solver at $\mathrm{tol} \in \{10^{-4}, 10^{-6}, 10^{-8}\}$, HOT, whose reduced flow model accepts the squared Euclidean cost only, and the convolutional Sinkhorn at $\varepsilon \in \{10^{-2}, 10^{-3}, 10^{-4}\}$.

\begin{figure}[t]
\centering
\includegraphics[width=\textwidth]{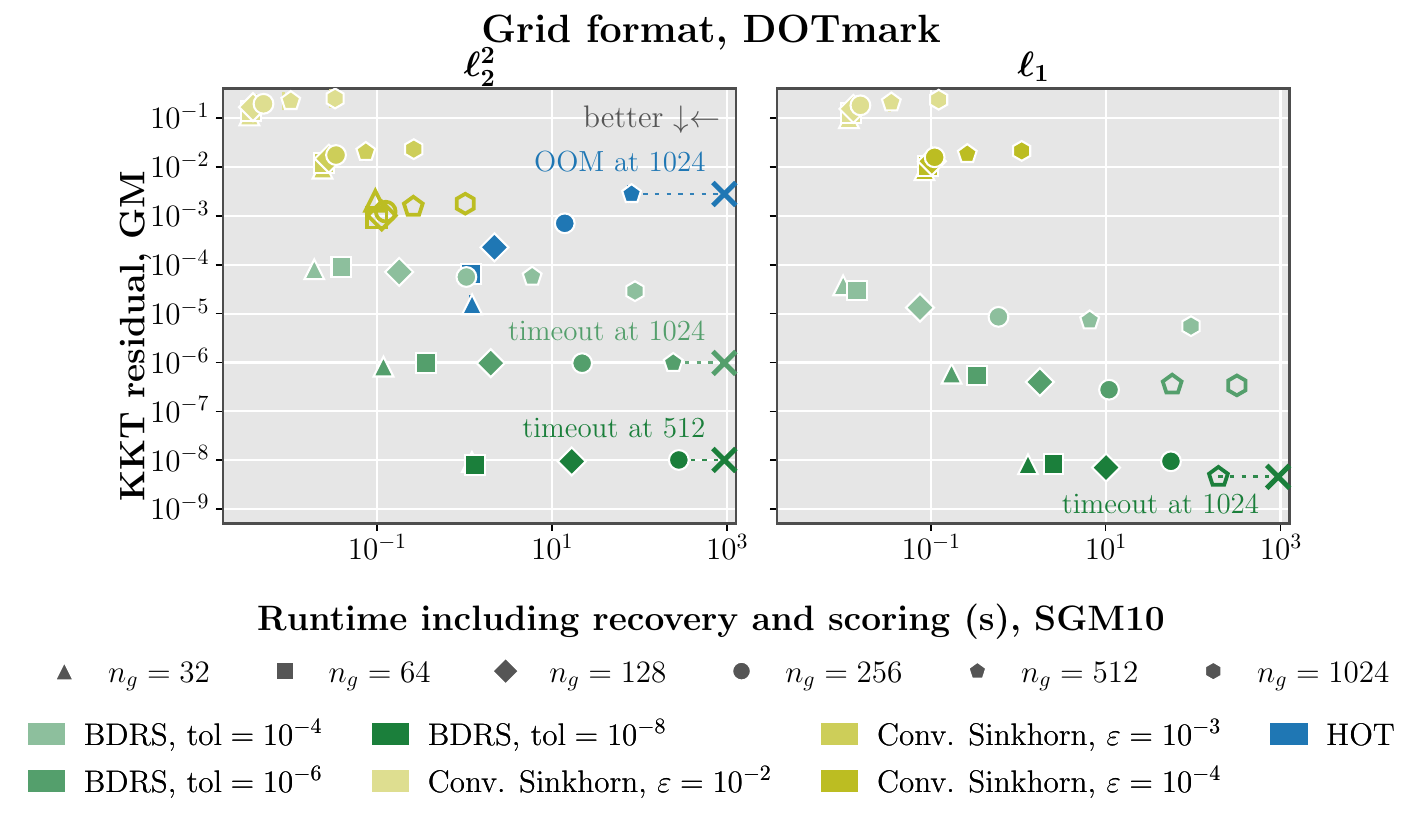}
\caption{The Grid solvers on DOTmark at $n_g = 32$ to $1024$ (the $n_g = 1024$ tier upsampled from the $n_g = 512$ images), with the $\ell_2^2$ cost (left) and the $\ell_1$ cost (right). Our Grid solver is drawn at $\mathrm{tol} \in \{10^{-4}, 10^{-6}, 10^{-8}\}$ and the convolutional Sinkhorn at $\varepsilon \in \{10^{-2}, 10^{-3}, 10^{-4}\}$. HOT accepts only the $\ell_2^2$ cost and therefore appears only on the left, where its certificate lines are hidden under its markers. An arm is drawn at a size where most of its solves return a finite solution within the time limit, and each point aggregates the $450$ co-solved instances of its size, with three exceptions, whose points are drawn hollow. (i) With the $\ell_1$ cost, our $\mathrm{tol} = 10^{-6}$ arm aggregates $418$ instances at $n_g = 512$, where the $32$ WhiteNoise instances time out, and at $n_g = 1024$, where those $32$ are not run, the $296$ of the $418$ that finish. (ii) With the $\ell_1$ cost, our $\mathrm{tol} = 10^{-8}$ arm at $n_g = 512$ is run on five pairs per class and aggregates the $42$ of the $50$ solves that finish. (iii) With the $\ell_2^2$ cost, the convolutional Sinkhorn at $\varepsilon = 10^{-4}$ aggregates $285$ to $366$ instances per size, because in OTT's scaling mode, which we run instead of its default log-domain mode, its iterations underflow to non-finite potentials on the rest. With the $\ell_1$ cost this happens on more than half of the instances at every size, so that arm is not drawn. A cross marks the size after an arm's last point. On the left at $n_g = 1024$, HOT runs out of device memory on every instance. The other three crosses are timeouts of our Grid solver, projected from its solve times at the next smaller size: $\mathrm{tol} = 10^{-6}$ at $n_g = 1024$ and $\mathrm{tol} = 10^{-8}$ at $n_g = 512$ with the $\ell_2^2$ cost, and $\mathrm{tol} = 10^{-8}$ at $n_g = 1024$ with the $\ell_1$ cost.}
\label{fig:exp23}
\end{figure}

Figure~\ref{fig:exp23} compares the Grid solvers on \textbf{DOTmark} at $n_g \in \{32, \dots, 1024\}$, one panel per cost. Against HOT, our solver is Pareto-dominant at every size at which HOT returns a solution ($n_g \le 512$). From $n_g = 128$ to $512$ the $\mathrm{tol} = 10^{-4}$ arm is both faster (by $12\times$ to $14\times$) and more accurate (HOT's $\KKT$ rising from $2.3 \times 10^{-4}$ to $2.8 \times 10^{-3}$ while ours stay below $10^{-4}$) than HOT. At $n_g=32$ and $n_g=64$, where HOT's $\KKT$ of $1.6 \times 10^{-5}$ and $6.5 \times 10^{-5}$ is below our native tolerance, the $\mathrm{tol} = 10^{-6}$ variant of our Grid solver is both faster ($0.12$ and $0.36$~s against $1.22$ and $1.19$~s) and $19$ and $67$ times more accurate. Lastly, at $n_g=1024$, HOT runs out of device memory while ours finishes in $89$~s for all 450 instances. At that size, only the $\mathrm{tol} = 10^{-4}$ arm completes all $450$ instances within the budget for both costs. With the $\ell_1$ cost, the $10^{-6}$ arm was run on $418$ instances, the $32$ WhiteNoise instances that time out at $n_g = 512$ not being run, and $296$ of them finish while $122$ time out; with the $\ell_2^2$ cost it is projected to time out. The $10^{-8}$ arm stops one size earlier still, at $n_g = 256$ with the $\ell_2^2$ cost and at $n_g = 512$ with the $\ell_1$ cost.

Convolutional Sinkhorn sacrifices accuracy --- as measured by its recovered primal--dual pair --- and reliability for fast runtimes. In particular, its $t_{\mathrm{cert}}$ barely moves with size, growing only $9.6$- to $13$-fold per arm while $N$ grows $1024$-fold. However, no entropic point with either cost falls below $\KKT = 9.1 \times 10^{-4}$, and at small values of $\varepsilon$, in the tested scaling-mode configuration, it stops being reliable. Concretely, the $\varepsilon = 10^{-2}$ and $10^{-3}$ variants solve all $450$ instances with both costs, but the KKT residuals their points achieve sit between $8.6 \times 10^{-3}$ and $2.5 \times 10^{-1}$, at least two orders of magnitude above ours at the same size. Its most accurate arm with $\ell_2^2$, the $\varepsilon = 10^{-4}$ variant, returns non-finite potentials on $19$ to $37$ percent of the instances with the squared Euclidean cost and on $65$ to $90$ percent with $\ell_1$, the latter of which is therefore not drawn on the right panel. By contrast, our Grid solver returns a certified solution on every instance it completes. Moreover, at $n_g = 32$, our Grid solver at $\mathrm{tol}=10^{-4}$ Pareto-dominates all convolutional Sinkhorn variants except $\varepsilon = 10^{-2}$ with the $\ell_2^2$ cost, and it dominates the $\varepsilon = 10^{-3}$ variant up to $n_g = 128$ with $\ell_1$ and the $\varepsilon = 10^{-4}$ variant up to $n_g = 64$ with $\ell_2^2$. In summary, our Grid solver at $\mathrm{tol}=10^{-4}$ is the only solver of~\eqref{eq:primal-problem} that solves all $450$ DOTmark instances at every size with both costs. Our Grid solver Pareto-dominates HOT at every size at which HOT returns a solution. Convolutional Sinkhorn's runtime scales better with problem size, but at the cost of accuracy and, at small $\varepsilon$ in the tested configuration, reliability.

\subsection{Ablation and robustness studies}\label{sec:exp-ablation}

In this subsection, we demonstrate the dual stall phenomenon of vanilla BDRS and perform ablation studies on the c-transform post-processing on the various Douglas--Rachford splitting-based solvers to demonstrate its effectiveness on BDRS. Then, we examine the robustness of BDRS with respect to the step size parameter $\eta$ and the termination check period $T_{\mathrm{chk}}$.

\subsubsection{Dual stall phenomenon and c-transform ablation} \label{sec:exp-ablation-stall}
\begin{figure}[t]
\centering
\includegraphics[width=\textwidth]{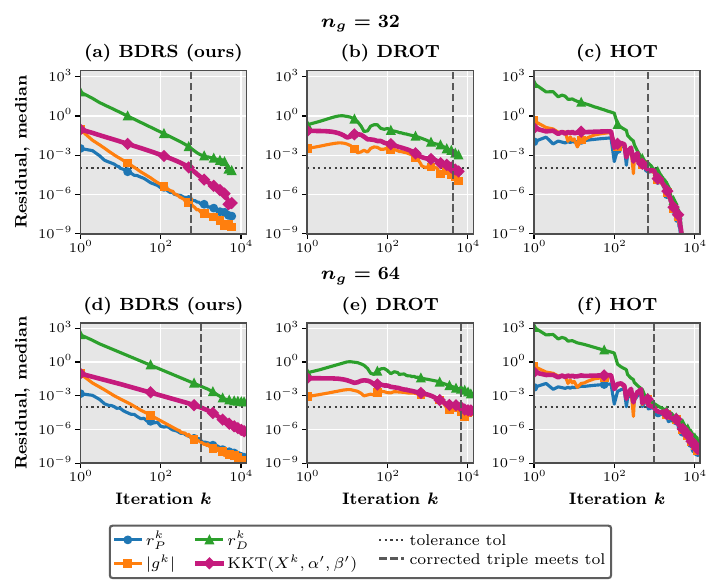}
\caption{The median of various residuals against iteration count for the three Douglas--Rachford solvers (BDRS, DROT, HOT), on the same $50$ DOTmark instances, with the $\ell_2^2$ cost and $\mathrm{tol}=10^{-4}$, at $n_g = 32$ in panels (a), (b), and (c) (resp. $n_g = 64$ in panels (d), (e), and (f)). Each solver was run for a fixed $6000$ (resp. $12000$) iterations with no stopping rule and scored at $101$ (resp. $163$) checks from iteration $1$. On all six panels, the blue, green, and orange curves are the medians of $r_P^k$, $r_D^k$, and $|g^k|$ for the triple $(X^k, \alpha^k, \beta^k)$, respectively. The magenta curve is the median KKT residual of the corrected triple $(X^k, \alpha', \beta')$, whose duals are the c-transformed \emph{corrected duals}. Note that the corrected duals are never fed back into the iterations and only serve as a termination check proxy.}
\label{fig:dual-stall}
\end{figure}

Section~\ref{sec:c-transform} motivates the c-transform post-processing by the observation that $r_D$ lags behind the other two residuals. Figure~\ref{fig:dual-stall} quantifies that lag on a subset of DOTmark for the three Douglas--Rachford solvers (BDRS, DROT, HOT). For BDRS and DROT, the KKT residual~\eqref{eq:kkt} of the native iterates is equal to $r_D^k$ at more than $98\%$ of the checkpoints. This is the aforementioned ``dual stall''. At the check where the corrected triple first meets the tolerance, the dual residual is $8.2 \times 10^{3}$ times (resp. $5.3 \times 10^{4}$ times) above the larger of the other two terms for BDRS at $n_g = 32$ (resp. $64$), and $21$ times (resp. $50$ times) for DROT. The dual stall is sharper for BDRS, as at $n_g = 32$ the primal residual and the native gap are $236$ and $540$ times below the tolerance, while the dual is $35$ times above it. As a result, the green curves (representing $r_D^k$) on panels (a), (b), (d), and (e) can also be interpreted as the native KKT residual, which always lags behind the corrected KKT residual, in magenta, demonstrating that the c-transform post-processing shortens the solve. More specifically, it shortens the solve by a median factor of $7.5$ for BDRS on the $42$ of the $50$ instances whose native triple reaches the tolerance at $n_g = 32$; at $n_g = 64$ the native triple reaches it on only $4$ of the $50$, so we take the length of the iteration window as the native crossing of the other $46$, which bounds the median factor below by $12$. For DROT the native triple never reaches the tolerance, so we can only lower bound the factor by $1.4$ (resp. $1.8$), which is computed as the length of the iteration window divided by the corrected crossing. For HOT, the largest term starts as $r_D^k$ but stops being the largest term at a median $k = 1000$ and $1600$ at the two sizes, at which point the native KKT residual and the corrected KKT residual become equal. Before that crossover, HOT's dual stall is of the same order as ours, $707$ times (resp. $2827$ times) the larger of the other two terms at $n_g = 32$ (resp. $n_g=64$) at $k = 1$. What separates HOT is that its stall then decays, to $182$ and $712$ times at $k = 100$ and to $2.4$ and $21$ times at $k = 500$, while ours grows to $9.2 \times 10^{3}$ and $3.5 \times 10^{4}$ at $k = 100$. As a result, the c-transform post-processing has a limited effect on HOT, shortening the solve by a median factor of $1.4$ and $1.5$.

\subsubsection{Robustness of BDRS with respect to hyperparameters} \label{sec:ablation-robustness} There are two hyperparameters for our solvers: the step size $\eta$ and the check period $T_{\mathrm{chk}}$. We provide a simple default rule for both, and leave the precise tuning of these parameters as future work. The deployed rules for both are described in Section~\ref{sec:experiment-setup}.

\begin{figure}[t]
\centering
\begin{subfigure}[t]{0.48\textwidth}
\centering
\includegraphics[width=\textwidth]{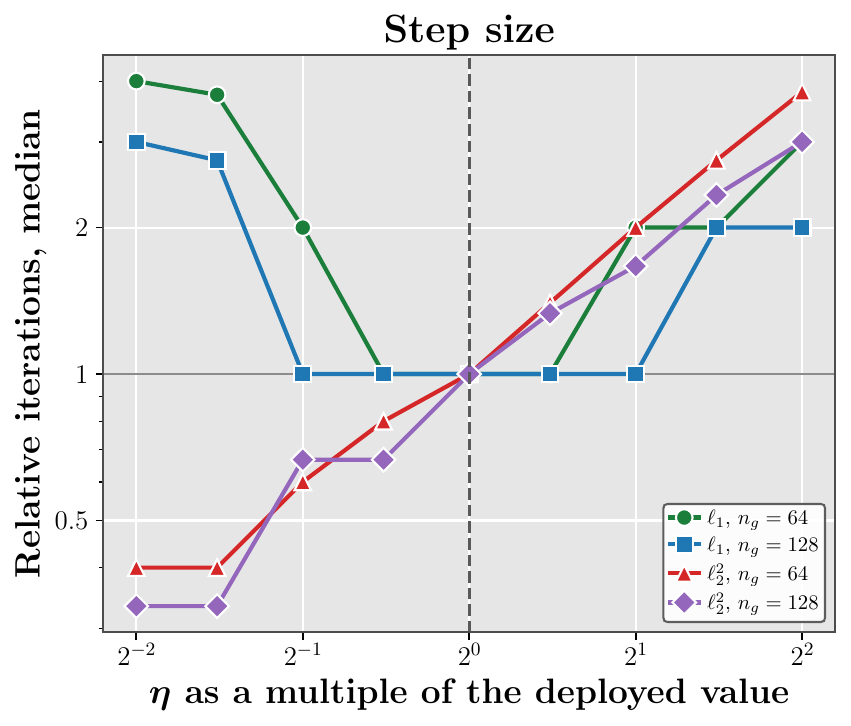}
\caption{Relative iteration count against step size $\eta$.}
\label{fig:robustness-eta}
\end{subfigure}\hfill
\begin{subfigure}[t]{0.48\textwidth}
\centering
\includegraphics[width=\textwidth]{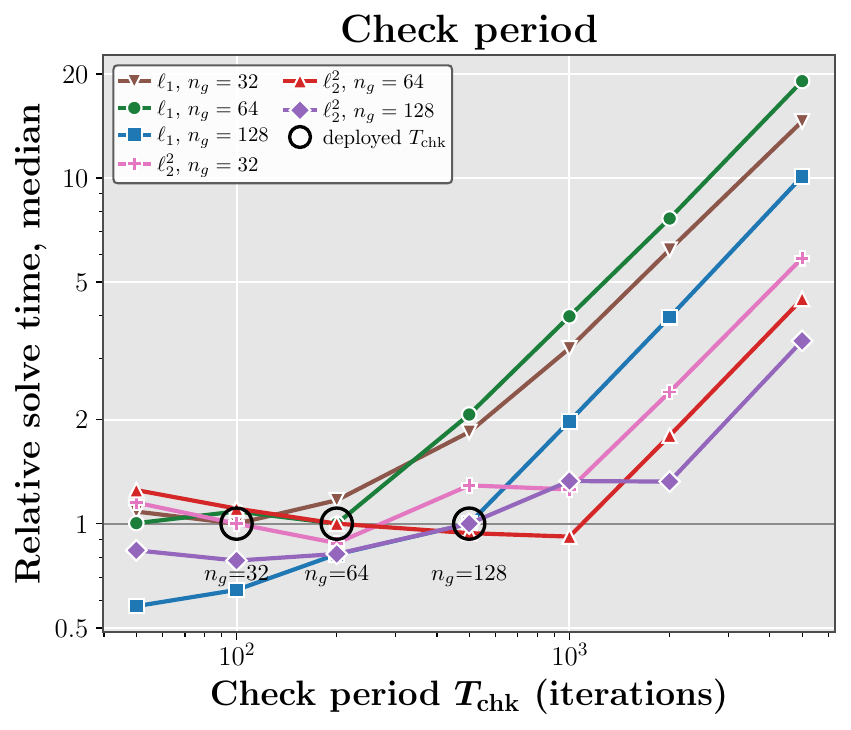}
\caption{Relative solve time against check period $T_{\mathrm{chk}}$.}
\label{fig:robustness-tchk}
\end{subfigure}
\caption{Sensitivity of our Grid solver to the (a) step size $\eta$ and (b) check period $T_{\mathrm{chk}}$, on DOTmark at $\mathrm{tol} = 10^{-4}$. Each solve's iteration count in (a), and its solve wall-clock time in (b), is divided by the same quantity of the same instance at the deployed value of the swept parameter, and the solid horizontal line marks a ratio of one. For both panels, one curve is plotted per choice of ground cost and resolution. (a) Resolutions $n_g \in \{64, 128\}$ were used, with $50$ instances each (five sampled pairs from all ten classes). The dashed vertical line marks the production value. (b) The production $\eta$ was used at $T_{\mathrm{chk}}$ ranging from $50$ to $5000$. Resolutions $n_g \in \{32, 64, 128\}$ were used, with $30$ instances each (three sampled pairs from all ten classes). The rings mark the deployed period of each resolution, $100$, $200$ and $500$ at $n_g = 32$, $64$ and $128$, where each curve is one by construction.}
\label{fig:robustness}
\end{figure}

Figure~\ref{fig:robustness} shows how the iteration count and solve time of our algorithm move as we perturb each parameter on a subset of DOTmark instances. In Figure~\ref{fig:robustness-eta}, we sweep the value of $\eta$ across a factor of sixteen, from a quarter of the deployed value to four times it. We find that our solvers are able to solve all instances in every perturbation, while the median iteration count stays inside $[0.33, 4.0]$ times the count that the solve took to solve the instances at the deployed value. Moreover, it is clear that the deployed $\eta$ is not the optimal choice for the $\ell_2^2$ instances of this subset, since we observe a linear trend between $\eta$ and the median of the iteration count required to solve them, but we emphasize that despite this suboptimality we still achieve the scalability demonstrated in Section~\ref{sec:exp-baselines}.

Next, in Figure~\ref{fig:robustness-tchk}, we sweep $T_{\mathrm{chk}}$ across a hundred-fold span, from $50$ to $5000$ iterations. Every perturbation solves every instance. The median solve time is flat from $T_{\mathrm{chk}} = 50$ to $T_{\mathrm{chk}} = 500$. At $T_{\mathrm{chk}} \ge 1000$, which represents an overly loose check period at the tested problem sizes, the solver mostly terminates at the first termination check, causing wasted iterations. Conversely, for $T_{\mathrm{chk}} = 50$, a period that is slightly shorter than preferred for the tested problem sizes, we indeed terminate at fewer total iterations, but at the cost of more frequent checks. More specifically, at $T_{\mathrm{chk}} = 50$, the pooled median solve takes about $16$ percent fewer iterations than at the deployed period and yet $5$ percent more time.

\section{Conclusion}\label{sec:conclusion}

In this paper, we presented three GPU implementations of the BDRS algorithm for solving the DOT problems with three data formats. We proposed a c-transform post-processing step that produces dual variables that are feasible up to machine precision. This procedure reduced the number of iterations required to satisfy the stopping criterion, without changing the BDRS trajectory. We conducted extensive numerical experiments to compare our GPU-accelerated BDRS with existing competitive solvers. The experiments demonstrate strong accuracy--time performance for the General and Grid implementations. The Point Cloud implementation is competitive on smaller instances and offers a runtime--accuracy trade-off against HALO and M3S at larger scales. Future work will distribute the computation across multiple GPUs, as suggested by the recent developments in multi-GPU mathematical programming solvers~\cite{li2026d,deza2026shardlp}.

\section*{Statements and Declarations}

\paragraph{Competing interests.} The authors have no relevant financial or non-financial interests to disclose.

\paragraph{Code and data availability.} The code of our three solvers and of the benchmark, including the scripts that build the datasets, is available at \url{https://github.com/mrhoggboss/BDRS-for-DOT} under the MIT license; the version used in this paper is release v2.1, \url{https://github.com/mrhoggboss/BDRS-for-DOT/releases/tag/v2.1}. DOTmark and ModelNet10 are public datasets, and the random dense instances and the point samples drawn from the ModelNet10 meshes are generated from fixed seeds given in the README of the release. The per-instance results are published with the release, together with the commands that redraw each figure from them.

\bibliographystyle{spmpsci}
\bibliography{ref}

\begin{thebibliography}{10}
\providecommand{\url}[1]{{#1}}
\providecommand{\urlprefix}{URL }
\expandafter\ifx\csname urlstyle\endcsname\relax
  \providecommand{\doi}[1]{DOI~\discretionary{}{}{}#1}\else
  \providecommand{\doi}{DOI~\discretionary{}{}{}\begingroup
  \urlstyle{rm}\Url}\fi

\bibitem{alaya2019screening}
Alaya, M.Z., B{\'e}rar, M., Gasso, G., Rakotomamonjy, A.: Screening {S}inkhorn
  algorithm for regularized optimal transport.
\newblock In: Advances in Neural Information Processing Systems, vol.~32 (2019)

\bibitem{altschuler2017near}
Altschuler, J., Weed, J., Rigollet, P.: Near-linear time approximation
  algorithms for optimal transport via {S}inkhorn iteration.
\newblock In: Advances in Neural Information Processing Systems, vol.~30, pp.
  1964--1974 (2017)

\bibitem{applegate2021pdlp}
Applegate, D., D{\'i}az, M., Hinder, O., Lu, H., Lubin, M., O'Donoghue, B.,
  Schudy, W.: Practical large-scale linear programming using primal-dual hybrid
  gradient.
\newblock In: Advances in Neural Information Processing Systems, vol.~34, pp.
  20243--20257 (2021)

\bibitem{auricchio2018computing}
Auricchio, G., Bassetti, F., Gualandi, S., Veneroni, M.: Computing
  {K}antorovich--{W}asserstein distances on d-dimensional histograms using
  (d+1)-partite graphs.
\newblock In: Advances in Neural Information Processing Systems, vol.~31 (2018)

\bibitem{bargetto2026improved}
Bargetto, R., Della~Croce, F., Scatamacchia, R.: An improved variant of the
  {Iterated Inside Out} algorithm for solving the optimal transport {DOTmark}
  instances.
\newblock Mathematical Programming Computation  (2026).
\newblock \doi{10.1007/s12532-026-00341-2}

\bibitem{bargetto2026iio}
Bargetto, R., Della~Croce, F., Scatamacchia, R.: {Iterated Inside Out}: a new
  exact algorithm for the transportation problem.
\newblock INFORMS Journal on Computing \textbf{38}(1), 39--52 (2026).
\newblock \doi{10.1287/ijoc.2024.0642}

\bibitem{bertsekas1982constrained}
Bertsekas, D.P.: Constrained Optimization and Lagrange Multiplier Methods.
\newblock Academic Press, New York (1982)

\bibitem{bertsekas1988auction}
Bertsekas, D.P.: The auction algorithm: a distributed relaxation method for the
  assignment problem.
\newblock Annals of Operations Research \textbf{14}(1), 105--123 (1988).
\newblock \doi{10.1007/BF02186476}

\bibitem{blanchard2021lse}
Blanchard, P., Higham, D.J., Higham, N.J.: Accurately computing the log-sum-exp
  and softmax functions.
\newblock IMA Journal of Numerical Analysis \textbf{41}(4), 2311--2330 (2021).
\newblock \doi{10.1093/imanum/draa038}

\bibitem{bobrutsky2025quantization}
Bobrutsky, J., Moscovich, A.: Quantization-based bounds on the {Wasserstein}
  metric.
\newblock arXiv preprint arXiv:2506.00976  (2025)

\bibitem{bonneel2011displacement}
Bonneel, N., van~de Panne, M., Paris, S., Heidrich, W.: Displacement
  interpolation using {L}agrangian mass transport.
\newblock ACM Transactions on Graphics \textbf{30}(6), 158:1--158:12 (2011).
\newblock \doi{10.1145/2070781.2024192}

\bibitem{burns2025lamp}
Burns, M.X., Liang, J.: Log-averaged mirror prox for fast, large-scale optimal
  transport in linear space.
\newblock arXiv preprint arXiv:2511.11359  (2025)

\bibitem{chen2022m3s}
Chen, Y., Li, C., Lu, Z.: Computing {Wasserstein}-$p$ distance between images
  with linear cost.
\newblock In: Proceedings of the IEEE/CVF Conference on Computer Vision and
  Pattern Recognition, pp. 519--528 (2022)

\bibitem{choquette2023hopper}
Choquette, J.: {NVIDIA Hopper H100 GPU}: Scaling performance.
\newblock IEEE Micro \textbf{43}(3), 9--17 (2023).
\newblock \doi{10.1109/MM.2023.3256796}

\bibitem{choquette2021a100}
Choquette, J., Gandhi, W., Giroux, O., Stam, N., Krashinsky, R.: {NVIDIA A100}
  tensor core {GPU}: Performance and innovation.
\newblock IEEE Micro \textbf{41}(2), 29--35 (2021).
\newblock \doi{10.1109/MM.2021.3061394}

\bibitem{courty2017optimal}
Courty, N., Flamary, R., Tuia, D., Rakotomamonjy, A.: Optimal transport for
  domain adaptation.
\newblock IEEE Transactions on Pattern Analysis and Machine Intelligence
  \textbf{39}(9), 1853--1865 (2017).
\newblock \doi{10.1109/TPAMI.2016.2615921}

\bibitem{cuturi2013sinkhorn}
Cuturi, M.: Sinkhorn distances: Lightspeed computation of optimal transport.
\newblock In: Advances in Neural Information Processing Systems, vol.~26 (2013)

\bibitem{cuturi2022ott}
Cuturi, M., Meng-Papaxanthos, L., Tian, Y., Bunne, C., Davis, G., Teboul, O.:
  Optimal transport tools ({OTT}): a {JAX} toolbox for all things
  {W}asserstein.
\newblock arXiv preprint arXiv:2201.12324 (2022)

\bibitem{dellamico2000assignment}
Dell'Amico, M., Toth, P.: Algorithms and codes for dense assignment problems:
  the state of the art.
\newblock Discrete Applied Mathematics \textbf{100}(1-2), 17--48 (2000).
\newblock \doi{10.1016/S0166-218X(99)00172-9}

\bibitem{deza2026shardlp}
Deza, A., Dey, S., Van~Hentenryck, P.: Distributed linear programming on {GPU}
  clusters at extreme scale.
\newblock arXiv preprint arXiv:2609.09108  (2026)

\bibitem{flamary2021pot}
Flamary, R., Courty, N., Gramfort, A., et~al.: {POT}: {Python} optimal
  transport.
\newblock Journal of Machine Learning Research \textbf{22}(78), 1--8 (2021)

\bibitem{hou2026badmm}
Hou, D., Toh, K.C.: A logarithmic-dual-update {Bregman} {ADMM} for optimal
  transport.
\newblock arXiv preprint arXiv:2610.08087  (2026)

\bibitem{ieee2019fp}
{IEEE}: {IEEE} standard for floating-point arithmetic.
\newblock IEEE Std 754-2019 (Revision of IEEE Std 754-2008) (2019).
\newblock \doi{10.1109/IEEESTD.2019.8766229}

\bibitem{kuhn1955hungarian}
Kuhn, H.W.: The {H}ungarian method for the assignment problem.
\newblock Naval Research Logistics Quarterly \textbf{2}(1--2), 83--97 (1955).
\newblock \doi{10.1002/nav.3800020109}

\bibitem{li2026d}
Li, H., Huang, Y., Liu, H., Ge, D., Ye, Y.: {D-PDLP}: Scaling {PDLP} to
  distributed multi-{GPU} systems.
\newblock arXiv preprint arXiv:2601.07628  (2026)

\bibitem{ling2007efficient}
Ling, H., Okada, K.: An efficient earth mover's distance algorithm for robust
  histogram comparison.
\newblock IEEE Transactions on Pattern Analysis and Machine Intelligence
  \textbf{29}(5), 840--853 (2007).
\newblock \doi{10.1109/TPAMI.2007.1058}

\bibitem{lu2025cupdlpx}
Lu, H., Peng, Z., Yang, J.: {cuPDLPx}: A further enhanced {GPU}-based
  first-order solver for linear programming.
\newblock arXiv preprint arXiv:2507.14051  (2025)

\bibitem{lu2024pdot}
Lu, H., Yang, J.: {PDOT}: A practical primal-dual algorithm and a {GPU}-based
  solver for optimal transport.
\newblock arXiv preprint arXiv:2407.19689  (2024)

\bibitem{bdrs-can-fail}
Ma, S.: {BDRS} can fail for matrix scaling and optimal transport.
\newblock Optimization Online (2026)

\bibitem{scaling-convergence}
Ma, S.: Convergence of {BDRS} as a matrix scaling algorithm.
\newblock Optimization Online (2026)

\bibitem{ma2025bregman}
Ma, S., Xiao, L., Zhao, R.: {Bregman} {Douglas}--{Rachford} splitting method.
\newblock arXiv preprint arXiv:2509.08739  (2025)

\bibitem{mai2021drot}
Mai, V.V., Lindb{\"a}ck, J., Johansson, M.: A fast and accurate splitting
  method for optimal transport: Analysis and implementation.
\newblock In: International Conference on Learning Representations (ICLR)
  (2022)

\bibitem{milakov2018online}
Milakov, M., Gimelshein, N.: Online normalizer calculation for softmax.
\newblock arXiv preprint arXiv:1805.02867  (2018)

\bibitem{munkres1957algorithms}
Munkres, J.: Algorithms for the assignment and transportation problems.
\newblock Journal of the Society for Industrial and Applied Mathematics
  \textbf{5}(1), 32--38 (1957).
\newblock \doi{10.1137/0105003}

\bibitem{nvidia2021a100datasheet}
{NVIDIA Corporation}: {NVIDIA A100} tensor core {GPU} datasheet.
\newblock
  \url{https://www.nvidia.com/content/dam/en-zz/Solutions/Data-Center/a100/pdf/nvidia-a100-datasheet-us-nvidia-1758950-r4-web.pdf}
  (2021).
\newblock Accessed 2026-09-21

\bibitem{orlin1997network}
Orlin, J.B.: A polynomial time primal network simplex algorithm for minimum
  cost flows.
\newblock Mathematical Programming \textbf{78}(2), 109--129 (1997).
\newblock \doi{10.1007/BF02614365}

\bibitem{paszke2019pytorch}
Paszke, A., Gross, S., Massa, F., Lerer, A., Bradbury, J., Chanan, G., Killeen,
  T., Lin, Z., Gimelshein, N., Antiga, L., Desmaison, A., K{\"o}pf, A., Yang,
  E., DeVito, Z., Raison, M., Tejani, A., Chilamkurthy, S., Steiner, B., Fang,
  L., Bai, J., Chintala, S.: {PyTorch}: An imperative style, high-performance
  deep learning library.
\newblock In: Advances in Neural Information Processing Systems 32 (2019)

\bibitem{peyre2019computational}
Peyr{\'e}, G., Cuturi, M.: Computational optimal transport: with applications
  to data science.
\newblock Foundations and Trends in Machine Learning \textbf{11}(5--6),
  355--607 (2019).
\newblock \doi{10.1561/2200000073}

\bibitem{rubner2000earth}
Rubner, Y., Tomasi, C., Guibas, L.J.: The earth mover's distance as a metric
  for image retrieval.
\newblock International Journal of Computer Vision \textbf{40}(2), 99--121
  (2000).
\newblock \doi{10.1023/A:1026543900054}

\bibitem{schmitzer2019stabilized}
Schmitzer, B.: Stabilized sparse scaling algorithms for entropy regularized
  transport problems.
\newblock SIAM Journal on Scientific Computing \textbf{41}(3), A1443--A1481
  (2019).
\newblock \doi{10.1137/16M1106018}

\bibitem{schrieber2016dotmark}
Schrieber, J., Schuhmacher, D., Gottschlich, C.: {DOTmark}--a benchmark for
  discrete optimal transport.
\newblock IEEE Access \textbf{5}, 271--282 (2017).
\newblock \doi{10.1109/ACCESS.2016.2639065}

\bibitem{schwinn2019heuristics}
Schwinn, J., Werner, R.: On the effectiveness of primal and dual heuristics for
  the transportation problem.
\newblock IMA Journal of Management Mathematics \textbf{30}(3), 281--303
  (2019).
\newblock \doi{10.1093/imaman/dpy011}

\bibitem{solomon2015convolutional}
Solomon, J., de~Goes, F., Peyr{\'e}, G., Cuturi, M., Butscher, A., Nguyen, A.,
  Du, T., Guibas, L.: Convolutional {W}asserstein distances: efficient optimal
  transportation on geometric domains.
\newblock ACM Transactions on Graphics \textbf{34}(4), 66:1--66:11 (2015).
\newblock \doi{10.1145/2766963}

\bibitem{tseng1993exponential}
Tseng, P., Bertsekas, D.P.: On the convergence of the exponential multiplier
  method for convex programming.
\newblock Mathematical Programming \textbf{60}(1), 1--19 (1993)

\bibitem{williams2009roofline}
Williams, S., Waterman, A., Patterson, D.: Roofline: an insightful visual
  performance model for multicore architectures.
\newblock Communications of the ACM \textbf{52}(4), 65--76 (2009).
\newblock \doi{10.1145/1498765.1498785}

\bibitem{wu2015shapenets}
Wu, Z., Song, S., Khosla, A., Yu, F., Zhang, L., Tang, X., Xiao, J.: {3D}
  {S}hape{N}ets: a deep representation for volumetric shapes.
\newblock In: IEEE Conference on Computer Vision and Pattern Recognition (CVPR)
  (2015)

\bibitem{xia2026halo}
Xia, W., Zhu, Y.N., Liang, J., Zhang, X.: A memory-efficient hierarchical
  algorithm for large-scale optimal transport problems.
\newblock In: The Fourteenth International Conference on Learning
  Representations (2026)

\bibitem{zhang2025hot}
Zhang, G., Gu, Z., Yuan, Y., Sun, D.: {HOT}: An efficient {Halpern}
  accelerating algorithm for optimal transport problems.
\newblock IEEE Transactions on Pattern Analysis and Machine Intelligence
  \textbf{47}(8), 6703--6714 (2025).
\newblock \doi{10.1109/TPAMI.2025.3564353}

\end{thebibliography}

\newpage
\appendix

\newpage
\section{Recovering primal--dual pairs from the baseline solvers}\label{app:duals}

The cross-solver comparison of Section~\ref{sec:experiments} scores every solver's returned solution by the KKT residual~\eqref{eq:kkt} on its \emph{corresponding} primal--dual triple $(X, \alpha, \beta)$ for the \emph{$C$-normalized} version of~\eqref{eq:primal-problem}. Each baseline solver terminates based on its own stopping rule, and for some baseline solvers, the mapping from their stopping rule to the KKT residual is not clear. For the various solvers, their final transport plan $X$ is recovered from:
\begin{enumerate}
    \item the array it returns (HALO, PDOT, cuPDLPx, log-domain Sinkhorn),
    \item the plan buffer of its final iterate (DROT, M3S), and
    \item the plan recovered from its final algorithm state (HOT, through HOT's own recovery rule, LAMP and Grid Sinkhorn).
\end{enumerate}
This appendix therefore records, for each baseline, only how the dual pair $(\alpha, \beta)$ is obtained from the solver's final state. For entropic solvers, we follow the convention of~\cite{lu2024pdot} and score the entropic potentials as the dual potentials of~\eqref{eq:primal-problem} directly, except for M3S, whose recovery we describe below.

\paragraph{DROT.} DROT~\cite{mai2021drot} applies Euclidean Douglas--Rachford splitting to~\eqref{eq:primal-problem}. Alongside the plan it maintains two vectors $\phi_1^k \in \R^m$ and $\phi_2^k \in \R^n$, which enter its iterations~\cite[Algorithm~1]{mai2021drot} as
\begin{equation*}
X^{k+1} = \bigl(X^k + \phi_1^k \oplus \phi_2^k - \sigma C\bigr)_+
\end{equation*}
with step size $\sigma > 0$. The authors of~\cite{mai2021drot} show that this pair implicitly plays the role of the dual variables of~\eqref{eq:dual-problem}: $(\phi_1^k/\sigma) \oplus (\phi_2^k/\sigma)$ converges to $\alpha^\star \oplus \beta^\star$, and the method's own stopping criteria are the primal residual, the dual residual and the duality gap at $(\phi_1^k/\sigma, \phi_2^k/\sigma)$. The provided implementation departs from this description: it stops once $\sqrt{\|X\mathbf{1}_n - r\|_2^2 + \|X^\top\mathbf{1}_m - c\|_2^2}$ falls below a given tolerance, and forms neither the dual residual nor the gap. We follow the implementation, with the tolerance $10^{-4}$ tested in single precision every $100$ iterations, so DROT's dual residual is computed only after its solve. Our recovery reads the two vectors at DROT's native stop and rescales,
\[
  (\alpha, \beta) = (\phi_1/\sigma, \phi_2/\sigma),
\]
with $\sigma = 2/(m+n)$ the step size that the implementation's example program passes.

\paragraph{HOT.} HOT~\cite{zhang2025hot} does not solve \eqref{eq:primal-problem} directly. On an $h \times w$ grid of pixels with a squared Euclidean ground cost, it solves the reduced flow LP of~\cite{auricchio2018computing}, in which every unit of mass moves in two parts: first within its row, then within its column. With a pixel's coordinates denoted as $(p, q)$ and the separable cost as $C_{ab} = \varphi_{\mathrm{c}}(|q - q'|) + \varphi_{\mathrm{r}}(|p- p'|)$ for $a = (p,q)$, $b = (p', q')$, the reduced LP has nonnegative variables $x^1_{p;q \to q'}$ (a horizontal move within row $p$, cost $\varphi_{\mathrm{c}}(|q - q'|)$) and $x^2_{q';p \to p'}$ (a vertical move within column $q'$, cost $\varphi_{\mathrm{r}}(|p - p'|)$), and three sets of constraints:
\begin{alignat}{2}
&\sum_{q'} x^1_{p;q\to q'} = r_{(p,q)}
&\qquad& \text{(consistency at source, one per pixel } (p,q)\text{)}, \label{eq:hot-src}\\
&\sum_{q} x^1_{p;q\to q'} = \sum_{p'} x^2_{q';p\to p'}
&& \text{(incoming} = \text{outgoing mass, one per node } (p,q')\text{)}, \label{eq:hot-cpl}\\
&\sum_{p} x^2_{q';p\to p'} = c_{(p',q')}
&& \text{(consistency at target, one per pixel } (p',q')\text{)}. \label{eq:hot-tgt}
\end{alignat}
This reduces the variable count from $h^2w^2$ to $hw(h+w)$. Let $(\lambda, \bar\alpha, \bar\beta)$ denote the duals of~\eqref{eq:hot-cpl},~\eqref{eq:hot-src}, and~\eqref{eq:hot-tgt}, respectively. The dual constraints of the reduced LP read
\begin{equation}\label{eq:hot-dual}
\bar\alpha_{(p,q)} + \lambda_{(p,q')} \le \varphi_{\mathrm{c}}(|q - q'|),
\qquad\qquad
\bar\beta_{(p',q')} - \lambda_{(p,q')} \le \varphi_{\mathrm{r}}(|p - p'|).
\end{equation}

\begin{prop}\label{prop:hot-lift}
If $(\lambda, \bar\alpha, \bar\beta)$ is feasible for the reduced dual, then $(\alpha, \beta) := (\bar\alpha, \bar\beta)$ is feasible for~\eqref{eq:dual-problem}, with the same dual objective value. Further, an optimal reduced dual $(\lambda^\star,\bar\alpha^\star, \bar\beta^\star)$ maps to an optimal $(\alpha^\star, \beta^\star)$ for~\eqref{eq:dual-problem}.
\end{prop}

\begin{proof}
Fix a source $a = (p, q)$ and a target $b = (p', q')$. Since we always transport mass horizontally then vertically, every route from $a$ to $b$ must pass through the node $(p, q')$. Adding the two constraints of \eqref{eq:hot-dual} along this route,
\[
\bar \alpha_a + \bar \beta_b \le \varphi_{\mathrm{c}}(|q - q'|) + \varphi_{\mathrm{r}}(|p - p'|) = C_{ab}.
\]
Every constraint of the original dual is thus the sum of exactly two reduced constraints. As a result, feasible reduced duals always yield feasible dual variables of~\eqref{eq:dual-problem}. Since the coupling block~\eqref{eq:hot-cpl} has a zero right-hand side, the reduced dual objective equals $r^\top\bar\alpha + c^\top\bar\beta$, which is the objective of~\eqref{eq:dual-problem} at $(\alpha, \beta)$. Lastly, since the reduced LP and~\eqref{eq:primal-problem} have equal optimal values~\cite[Thm.~1 and Cor.~1]{auricchio2018computing}, strong duality on both the reduced LP and~\eqref{eq:dual-problem} makes $(\alpha, \beta)$ optimal for~\eqref{eq:dual-problem} whenever the reduced dual is optimal for the reduced LP.
\end{proof}

With $c_{\mathrm{LP}}$ and $b$ denoting the reduced LP's cost vector and right-hand side, HOT's preconditioning procedure internally rescales them by $1/c_s$ and $1/b_s$, respectively, with $c_s = \|c_{\mathrm{LP}}\|_2 + 1$ and $b_s = \|b\|_2 + 1$. A right-hand-side rescaling leaves the dual solution unchanged while a cost rescaling scales it linearly, so the recovery mapping multiplies HOT's returned dual vector by $c_s$ alone. Composing this with Proposition~\ref{prop:hot-lift} gives the mapping used in our code, $(\lambda, \bar \alpha, \bar \beta) \mapsto (c_s \bar \alpha, c_s \bar \beta)$.

\paragraph{HALO.} HALO~\cite{xia2026halo} is a coarse-to-fine restricted-support method. Both marginal distributions $r$ and $c$ are coarsened through a procedure that merges pairs of supports next to each other. At each level of coarseness, the method solves the LP restricted to a sparse candidate active set~\cite[Def.~1]{xia2026halo}, with the marginal constraints of that level's coarsened supports. cuPDLPx~\cite{lu2025cupdlpx} is used as the inner LP solver. Then, HALO prices the excluded entries with the resulting dual pair $(u, v) \in \R^m \times \R^n$, appending those of most negative reduced cost $C_{ij} - u_i - v_j$ and re-solving. At the finest level, the restricted LP retains the original marginal constraints, so its dual variables provide potentials on the original supports. Therefore, at HALO's native stop, our recovery map reads the finest-level dual variables, reversing any additional scaling that HALO performs internally. In fact, HALO applies scaling only in its Grid mode, where $C_{ij}^{\text{HALO, Grid}} = ((p-p')^2 + (q-q')^2)/n_g^2$. There is no scaling in its Point Cloud (Tree) mode. Therefore, our recovery map from HALO's Grid mode is $(u^{\text{Grid}}, v^{\text{Grid}}) \mapsto (n_g^2 u^{\text{Grid}} / (2(n_g-1)^2),n_g^2 v^{\text{Grid}}/ (2(n_g-1)^2))$, while the map for its Point Cloud (Tree) mode is $(u^{\text{Tree}}, v^{\text{Tree}}) \mapsto (u^{\text{Tree}}, v^{\text{Tree}})$.

\paragraph{LAMP.} LAMP~\cite{burns2025lamp} runs primal--dual mirror prox on~\eqref{eq:primal-problem} in a log-averaged form that keeps only dual variables, and certifies its own termination with a vector $\theta \in \R^n$. We run its dense path on the stored $C$ (General) and its matrix-free path on point coordinates (Point Cloud), the latter on coordinates divided by $C_{\max}$ (by $\sqrt{C_{\max}}$ for $\ell_2^2$), so that both paths solve the normalized cost. We score the pair its stopping test uses: $\beta := -2W_\infty \theta$, where $W_\infty := \max_{ij} C_{ij}$ equals $1$ on our normalized instances, and $\alpha$ the c-transform of $\beta$, $\alpha_i := \min_j(C_{ij} - \beta_j)$, which is dual feasible by construction. Note that LAMP runs on the marginals of its upstream preprocessing, $\tilde r := (r + 10^{-6}\mathbf{1}_m)/\|r + 10^{-6}\mathbf{1}_m\|_1$ and $\tilde c$ likewise (Section~\ref{sec:experiment-setup}), so its KKT residual is bounded below, up to LAMP's own residual against $(\tilde r, \tilde c)$, by the displacement of $(\tilde r, \tilde c)$ from $(r, c)$.

\paragraph{PDOT and cuPDLPx.} Both PDOT~\cite{lu2024pdot} and cuPDLPx~\cite{lu2025cupdlpx} return the dual of the original problem that is passed to them. Since we normalize $C$ before passing it to them, we can take the dual variables returned by both PDOT and cuPDLPx as-is. When cuPDLPx's presolve removes the constraint of a zero-mass row or column, its postsolve assigns that dual entry a value of its own, which enters $r_D$.

\paragraph{Log-domain Sinkhorn, Grid Sinkhorn, and M3S.}
Log-domain Sinkhorn~\cite{schmitzer2019stabilized}, as implemented in POT~\cite{flamary2021pot}, maintains the log-scalings $(u, v)$ of the plan $X = \exp\bigl(-C/\varepsilon + u \oplus v\bigr)$, where $\varepsilon$ is the regularization parameter. It iterates as follows:
\begin{equation}\label{eq:sinkhorn-halfsteps}
u_i \gets \log r_i - \LSE_j\bigl(-C_{ij}/\varepsilon + v_j\bigr),
\qquad
v_j \gets \log c_j - \LSE_i\bigl(-C_{ij}/\varepsilon + u_i\bigr).
\end{equation}
Setting $(\alpha, \beta) := (\varepsilon u, \varepsilon v)$ gives
\begin{equation}\label{eq:entropic-identity}
\alpha_i + \beta_j - C_{ij} = \varepsilon \log X_{ij}.
\end{equation}
Let us denote $H(X) := -\sum_{ij} X_{ij}\log X_{ij}$. Throughout this paragraph, $r_P$, $r_D$ and $g$ are the primal residual, the dual residual and the duality gap of the \emph{original} problem~\eqref{eq:primal-problem}. At a point of zero mass, each of the three recoveries of this paragraph gives the potential $-\infty$, which we keep: it enters $r_D$ as $(-\infty)_+ = 0$ and the dual objective through the convention of Section~\ref{sec:prelim}. Firstly, after any row (resp. column) update of~\eqref{eq:sinkhorn-halfsteps}, the corresponding marginal sum agrees exactly with $r$ (resp. $c$). Since the entries are nonnegative and $r$ and $c$ are probability vectors, each entry is bounded by its own marginal, $X_{ij} \le r_i \le 1$ (resp. $X_{ij} \le c_j \le 1$). Hence $\log X_{ij} \le 0$, and~\eqref{eq:entropic-identity} gives $\alpha \oplus \beta \le C$ entrywise. This means that the recovered dual variables $(\alpha, \beta)$ are always feasible for~\eqref{eq:dual-problem}.  Secondly, at a primal-feasible iterate, where $r = X\mathbf{1}_n$ and $c = X^\top\mathbf{1}_m$,~\eqref{eq:entropic-identity} gives
\[
r^\top\alpha + c^\top\beta = \textstyle\sum_{ij} X_{ij}(\alpha_i + \beta_j)
 = \langle C, X\rangle + \varepsilon \textstyle\sum_{ij} X_{ij}\log X_{ij},
\]
so that the duality gap is
\[
g = \langle C, X\rangle - r^\top\alpha - c^\top\beta = -\varepsilon \textstyle\sum_{ij} X_{ij}\log X_{ij} = \varepsilon H(X).
\]
Note that this gap is a property of the recovered primal--dual pair and need not reflect the objective error of the plan $X$ in~\eqref{eq:primal-problem}.

Grid Sinkhorn~\cite{solomon2015convolutional}, as implemented in OTT~\cite{cuturi2022ott}, runs the same iteration, which we run in OTT's scaling (kernel) mode, since in OTT 0.6.0 its default log-domain mode returns wrong values on the instances with a point of zero mass, $20$ to $33$ percent of our DOTmark instances at each size. OTT's routine for the marginals of the plan is wrong on the same instances in either mode, so we compute $X\mathbf{1}_n$ and $X^\top\mathbf{1}_m$ for $r_P$ from $(f, g)$ by two one-dimensional $\LSE$ passes over the separable cost, as for $\langle C, X\rangle$. Its potentials $(f, g)$ are already in cost units, so $(\alpha, \beta) := (f, g)$, and~\eqref{eq:entropic-identity} and both consequences apply verbatim. We measure $r_D$ of $(f, g)$ over all $N^2$ entries for $n_g \le 128$, where it is exactly zero on every solve for the instances reported in the figure. For larger $n_g$ we do not measure it, and record $r_D = 0$ as inferred from the first consequence, which holds in exact arithmetic.

M3S~\cite{chen2022m3s} is a multiscale $\varepsilon$-scaling Sinkhorn. Unlike the two Sinkhorn solvers above, it carries the floored marginals $(\tilde r, \tilde c)$ of Section~\ref{sec:prelim} in the exponential,
\begin{equation}\label{eq:m3s-plan}
X_{ij} = \tilde r_i \tilde c_j \exp\bigl((\hat\alpha_i + \hat\beta_j - C_{ij})/\varepsilon\bigr),
\end{equation}
with $(\hat\alpha, \hat\beta)$ being its finest-level potentials and $\varepsilon$ the last value of its $\varepsilon$ ladder, whose hard-wired target is $10^{-1}$ in squared-pixel units. Taking logs in~\eqref{eq:m3s-plan} and multiplying both sides by $\varepsilon$,
\[
\varepsilon\log X_{ij} = \bigl(\hat\alpha_i + \varepsilon\log \tilde r_i\bigr) + \bigl(\hat\beta_j + \varepsilon\log \tilde c_j\bigr) - C_{ij},
\]
so the plan's own pair $(\hat\alpha + \varepsilon\log\tilde r, \hat\beta + \varepsilon\log\tilde c)$ satisfies~\eqref{eq:entropic-identity}. Our recovery map instead uses the original marginals, $(\hat \alpha, \hat \beta) \mapsto (\hat\alpha + \varepsilon\log r, \hat\beta + \varepsilon\log c)$, which shifts that pair by $\varepsilon\log(r/\tilde r)$ and $\varepsilon\log(c/\tilde c)$, so~\eqref{eq:entropic-identity} holds for it only up to these terms. Accordingly, at a plan that meets the original marginals, the gap of the recovered pair is $\varepsilon H(X) - \varepsilon D_{\mathrm{KL}}(r\|\tilde r) - \varepsilon D_{\mathrm{KL}}(c\|\tilde c)$ rather than $\varepsilon H(X)$, where $D_{\mathrm{KL}}(r\|\tilde r) := \sum_i r_i\log(r_i/\tilde r_i)$. Moreover, both potentials are then rescaled from M3S's native squared-pixel cost units by $1/(2(n_g-1)^2)$, the same factor as its objective. Therefore, the dual recovery mapping we implement is
\[
(\hat \alpha, \hat \beta) \mapsto \Bigl(\frac{\hat\alpha + \varepsilon\log r}{2(n_g-1)^2}, \frac{\hat\beta + \varepsilon\log c}{2(n_g-1)^2}\Bigr).
\]
Lastly, M3S truncates its plan. More specifically, outside of a sparse support set, the plan is set to zero by construction rather than driven to zero asymptotically, so the identity constrains $(\alpha, \beta)$ only on that support and says nothing about the remaining entries. Dual feasibility there is therefore not automatic, and our experiments measure and report $r_D$ as part of the KKT residual~\eqref{eq:kkt} on the delivered pair like any other.

\section{Outcomes by solver and resolution}\label{app:outcomes}

Table~\ref{tab:outcomes} complements the co-solved aggregates of Section~\ref{sec:exp-baselines}: for every solver and DOTmark resolution, it gives the share of attempted solves that are scored and the share whose KKT residual is at most $10^{-4}$.

\begin{table}[t]
\caption{Outcomes on DOTmark by solver and resolution. In each cell, the left number is the percentage of attempted solves that are scored (a finite KKT residual within the limits of the protocol) and the right number is the percentage whose KKT residual is at most $10^{-4}$, over every cost the solver runs. Our solvers are run at $\mathrm{tol} = 10^{-4}$. A size the solver refuses before launch counts as attempted and not scored; --- marks a size that was not run.}
\label{tab:outcomes}
\centering
\footnotesize
\setlength{\tabcolsep}{3pt}
\begin{tabular}{lcccccc}
\hline\noalign{\smallskip}
Solver & \multicolumn{6}{c}{$n_g$} \\
 & 32 & 64 & 128 & 256 & 512 & 1024 \\
\noalign{\smallskip}\hline\noalign{\smallskip}
\multicolumn{7}{l}{\emph{Grid format}} \\
\quad Grid BDRS (ours) & 100/100 & 100/100 & 100/100 & 100/100 & 100/100 & 100/100 \\
\quad HOT & 100/100 & 100/96 & 100/$<$1 & 100/0 & 100/0 & 0/0 \\
\quad Conv.\ Sinkhorn, $\varepsilon = 10^{-2}$ & 100/0 & 100/0 & 100/0 & 100/0 & 100/0 & 100/0 \\
\quad Conv.\ Sinkhorn, $\varepsilon = 10^{-3}$ & 100/0 & 100/0 & 100/0 & 100/0 & 100/0 & 100/0 \\
\quad Conv.\ Sinkhorn, $\varepsilon = 10^{-4}$ & 58/0 & 38/0 & 40/0 & 37/0 & 39/0 & 39/$<$1 \\
\noalign{\smallskip}
\multicolumn{7}{l}{\emph{Point Cloud format}} \\
\quad Point Cloud BDRS (ours) & 100/100 & 100/100 & 100/100 & $>$99/$>$99 & --- & --- \\
\quad HALO & 100/$>$99 & 100/58 & 100/5 & 100/1 & 100/$<$1 & 0/0 \\
\quad M3S & 100/0 & 100/0 & 98/69 & 82/63 & 42/40 & 0/0 \\
\quad LAMP & 100/97 & 100/59 & 100/23 & 100/0 & --- & --- \\
\noalign{\smallskip}
\multicolumn{7}{l}{\emph{General format}} \\
\quad General BDRS (ours) & 100/100 & 100/100 & 100/100 & 100/100 & --- & --- \\
\quad DROT & 100/0 & 100/0 & 100/0 & 0/0 & --- & --- \\
\quad PDOT & 100/$<$1 & 98/$<$1 & $>$99/0 & 0/0 & --- & --- \\
\quad cuPDLPx & 100/0 & 100/0 & 0/0 & --- & --- & --- \\
\quad Sinkhorn, $\varepsilon = 10^{-2}$ & 100/0 & 100/0 & 100/0 & 0/0 & --- & --- \\
\quad Sinkhorn, $\varepsilon = 10^{-3}$ & 100/0 & 100/0 & 100/0 & 0/0 & --- & --- \\
\quad Sinkhorn, $\varepsilon = 10^{-4}$ & 100/0 & 100/0 & 100/0 & 0/0 & --- & --- \\
\quad LAMP & 100/80 & 100/36 & 100/$<$1 & 0/0 & --- & --- \\
\noalign{\smallskip}\hline
\end{tabular}
\end{table}

\end{document}